\documentclass[11pt,reqno]{amsart}

\usepackage{amsmath, amsfonts, amsthm, amssymb}

\usepackage{color}

\allowdisplaybreaks[4]

\newtheorem{Theorem}{Theorem}[section]
\newtheorem{Lemma}{Lemma}[section]
\newtheorem{Proposition}{Proposition}[section]

\theoremstyle{definition}
\newtheorem{Definition}{Definition}[section]

\theoremstyle{remark}
\newtheorem{Remark}{Remark}[section]

\numberwithin{equation}{section}

\renewcommand{\r}{\rho}

\def\vr{\varrho}

\renewcommand{\u}{{\bf u}}

\newcommand{\R}{{\mathbb R}}
\newcommand{\Dv}{{\rm div}}

\newcommand{\dl}{\delta}

\def\f{\frac}
\renewcommand{\O}{\Omega}

\def\D{\Delta }

\def\hf1{^\f{1}{1-\xi^2}}

\def\vr{\varrho}

\def\be{\begin{equation}}
\def\en{\end{equation}}
\def\bs{\begin{split}}
\def\es{\end{split}}

\renewcommand{\d}{{\bf d}}

\begin{document}

\title[Global Solution to Compressible Flow of
Liquid Crystals in 3D]{Global Solution to the Three-Dimensional Compressible Flow of
Liquid Crystals}

\author{Xianpeng Hu AND Hao Wu}

\address{Courant Institute of Mathematical
Sciences, New York University, New York, NY 10012.}
\email{xianpeng@cims.nyu.edu}
\address{School of Mathematical Sciences and Shanghai Key Laboratory for Contemporary Applied Mathematics, Fudan University, Shanghai 200433, PR China}
\email{haowufd@yahoo.com}

\date{\today}
\thanks{X.-P Hu was partially supported by DMS-1108647. H. Wu was partially supported by NSF of
China 11001058, SRFDP and ``Chen Guang" project supported by Shanghai
Municipal Education Commission and Shanghai Education Development
Foundation.}

\begin{abstract}
The Cauchy problem for the three-dimensional compressible flow of
nematic liquid crystals is considered. Existence and uniqueness of the
global strong solution are established in critical Besov spaces provided that the initial datum is close to an equilibrium state $(1,{\bf 0},
\hat{\d})$ with a constant vector $\hat{\d}\in S^2$. The global
existence result is proved via the local well-posedness and uniform estimates for proper linearized systems with convective terms.
\end{abstract}

\maketitle

\section{Introduction}
Liquid crystals are substances that exhibit a phase of matter that
has properties between those of a conventional liquid, and those
of a solid crystal. The three-dimensional flow of nematic liquid
crystals can be governed by the following system of partial
differential equations \cite{Gen,Lin2}:

\begin{subequations}\label{e1}
\begin{align}
& \f{\partial\r}{\partial t}+\Dv(\r\u)=0\quad
\label{e10}\\
&\f{\partial(\r\u)}{\partial
t}+\Dv(\r\u\otimes\u)-\mu\D\u-(\mu+\lambda)\nabla\Dv\u+\nabla
P(\r)\\
&\quad=-\xi\Dv\left(\nabla \d\odot\nabla\d-\f{1}{2}|\nabla \d|^2I\right)\quad\label{e11}\\
&\f{\partial\d}{\partial
t}+\u\cdot\nabla\d=\theta\left(\D\d+|\nabla\d|^2\d\right),\quad\label{e12}
\end{align}
\end{subequations}
where $\rho\in \mathbb{R}$ is the density function of the fluid, $\u\in\R^3$ is the velocity and $\d\in S^2$ represents the director
field for the averaged macroscopic molecular orientations. The scalar function $P\in\R$ is the pressure, which is an increasing and convex function in $\r$. They all depend on the spatial variable
 $x=(x_1,x_2,x_3)\in\R^3$ and the time variable $t>0$.
The constants $\mu$ and $\lambda$ are shear viscosity and the bulk viscosity coefficients of the fluid respectively that satisfy the physical assumptions $\mu>0, 2\mu+3\lambda\geq 0$. The constants
$\xi>0, \theta>0$ stand for the competition
between the kinetic energy and the potential energy, and the
microscopic elastic relaxation time (or the Debroah number) for the
molecular orientation field, respectively. The symbol $\otimes$ denotes the Kronecker tensor
product such that $\u\otimes \u= (\u_i \u_j)_{1\leq i,j\leq 3}$ and the $\nabla\d\odot\nabla\d$ denotes a matrix
whose $ij$-th entry ($1\leq i,j\leq 3$) is $\partial_{x_i} \d \cdot \partial_{x_j}\d$. $I$ is the $3\times 3$ identity matrix.
To complete the system \eqref{e1}, the initial data are given by
\begin{equation}\label{ic}
\r|_{t=0}=\r_0(x),\quad\u|_{t=0}=\u_0(x),\quad \d|_{t=0}=\d_0(x), \ \ \text{with}\  \d_0\in S^2.
\end{equation}

Roughly speaking, the system \eqref{e1} is a
coupling between the compressible Navier--Stokes equations and a
transported heat flow of harmonic maps into $S^2$. It is
a macroscopic continuum description of the evolution for the liquid crystals of nematic type under
the influence of both the flow field $\u$, and the macroscopic description of the
microscopic orientation configurations $\d$ of (rod-like) liquid crystals.

The hydrodynamic theory of liquid crystals in nematic case has
been established by Ericksen and Leslie, see \cite{Eri, Eri2,
Leslie1, Leslie2}. Since then, the mathematical theory is still
progressing and the study of the full Ericksen--Leslie model
presents relevant mathematical difficulties. In \cite{Lin2}, Lin
introduced a simplification of the general Ericksen--Leslie system
that keeps many of the mathematical difficulties of the original
system by using a Ginzburg--Landau approximation to relax the
nonlinear constraint $\d\in S^2$. Later in \cite{LL}, Lin and Liu
showed the global existence of weak solutions and smooth solutions
for that approximation system. For more results on the
approximation system, we refer to \cite{LL96, S01, W10}. Recently,
Hong \cite{Hong} and Lin--Lin--Wang \cite{LLW} showed
independently the global existence of weak solution of an
incompressible model of system \eqref{e1} in two dimensional
space. Moreover, in \cite{LLW}, the regularity of solutions except
for a countable set of singularities whose projection on the time
axis is a finite set had been obtained. In the recent work
\cite{WC}, Wang established a global well-posedness theory for the
incompressible liquid crystals for rough initial data, provided
that
$$\|\u_0\|_{BMO^{-1}}+[\d_0]_{BMO}\le \varepsilon_0$$
for some $\varepsilon_0>0$. Note that the relationship between
$BMO^{-1}$ and $\dot{H}^{\f12}$ is (see \cite{BP})
$$\dot{H}^{\f12}\hookrightarrow L^3\hookrightarrow
B^{-1+\f{3}{p}}_{p,\infty}\hookrightarrow BMO^{-1}, \quad \text{for}\  3\le
p<\infty.$$

Concerning the compressible case, local existence of unique strong
solutions of \eqref{e1} was proved provided that the initial data
$\r_0, \u_0, \d_0$ are sufficiently regular and satisfy a natural
compatibility condition in a recent work \cite{HWW2}. A criterion
for possible breakdown of such a local strong solution at finite
time was given in terms of blow up of the $L^\infty$-norms of
$\rho$ and $\nabla \d$. In \cite{WC1}, an alternative blow-up
criteria was derived in terms of the $L^\infty$-norms of
$\nabla\u$ and $\nabla\d$. The global existence of weak solutions
to \eqref{e1} with large initial data is still an outstanding open
problem for high dimensions. By so far, results in one space
dimension have been obtained in \cite{DWW, DLWW}, and authors in
\cite{JJW} consider a multidimensional version with small energy.

In this paper, we are interested in the existence and uniqueness of global strong
solutions to the Cauchy problem of \eqref{e1} with initial data \eqref{ic} in the three dimensional space. It is
difficult to find a functional space such that the system
\eqref{e1}--\eqref{ic} is well-posed globally in time. To this
end, we notice that the system \eqref{e1} is invariant under the
following transformations
\begin{equation}\label{si}
\tilde{\r}=\r(l^2t, lx),\quad\tilde{\u}=l\u(l^2t,lx), \quad
\tilde{\d}=\d(l^2t, lx)
\end{equation}
with the modification of the pressure $\tilde{P}=l^2P$. A critical
space is a space in which the norm is invariant under the scaling
$$(\tilde{e},\tilde{\mathbf{f}},\tilde{\mathbf{g}})(x)=(e(lx),l \mathbf{f}(lx), \mathbf{g}(lx)).$$ The scaling invariance \eqref{si} reminds us of a
similar property of three dimensional incompressible Navier--Stokes equations,
which provides a well-known global existence of solutions with
small data in the homogeneous Sobolev space $\dot{H}^{\f{1}{2}}$ (see
\cite{Kato64}). Motivated by this observation, we aim at a global
well-posedness of the system \eqref{e1}--\eqref{ic} with small
initial data in a functional framework where the function space
for the velocity $\u$ is similar to $\dot{H}^{\f{1}{2}}$. According to
the scaling \eqref{si}, the regularity of the density $\r$ and the
director field $\d$ is one order higher than that of the velocity $\u$, and
hence a function space which is similar to $\dot{H}^{\f32}$ would be
a candidate. Unfortunately, the function space
$\dot{H}^{\f{3}{2}}$ does not turn to be a good candidate for $\r$
and $\d$, since bounds of the director field $\d$ in
$\dot{H}^{\f{3}{2}}$ do not automatically imply the $L^\infty$
bound of $\d$. To overcome this difficulty, inspired by
\cite{RD2} for the compressible Navier--Stokes equations, it seems more natural to work in the framework of
homogeneous Besov space $B^{\f{3}{2}}_{2,1}$ since
$B^{\f32}_{2,1}$ is continuously embedded into $L^\infty$.
Furthermore, as in \cite{RD2}, the different dissipative
mechanisms of low frequencies and high frequencies inspire us to
deal with $\r$ and $\d$ in
$\tilde{B}^{\f12,\f32}_{2,1}=B^{\f32}_{2,1}\cap B^{\f12}_{2,1}.$

The rest of this paper is organized as follows. In Section 2, we
will give a series of fundamental properties of the Besov's spaces. In
Section 3, we reformulate the system \eqref{e1}--\eqref{ic} and state our main
result (Theorem \ref{MT}). The main goal of Section 4 is to prove uniform estimates for linearized systems to \eqref{e1}, while the global existence is
obtained in Section 5. In Section 6, the uniqueness of global strong solution is verified.

\bigskip

\section{Preliminaries}

Throughout this paper,  we use $C$ for a generic constant, and
denote  $A\le CB$ by  $A\lesssim B$. The notation $A\thickapprox
B$ means that $A\lesssim B$ and $B\lesssim A$. Also we use
$(\alpha_q)_{q\in\mathbb{Z}}$ to denote a sequence such that
$\sum_{q\in\mathbb{Z}}\alpha_q\le 1$. $(f|g)$ denotes the inner
product of two functions $f, g$ in $L^2(\R^N)$. The standard
summation notation over the repeated index will be adopted in the remaining part of this
paper.

In order to state our existence result, we introduce some
functional spaces and explain the notations. Let $T>0$,
$r\in[0,\infty]$ and $X$ be a Banach space. We denote by
$\mathcal{M}(0,T;X)$ the set of measurable functions on $(0,T)$
valued in $X$. For $f\in \mathcal{M}(0,T; X)$, we define
$$\|f\|_{L^r_T(X)}=\left(\int_0^T\|f(\tau)\|_X^rd\tau\right)^{\f{1}{r}}\textrm{
if }r<\infty,$$
$$\|f\|_{L^\infty_T(X)}=\sup \textrm{ess}_{\tau\in(0,T)}\|f(\tau)\|_X.$$
Denote $L^r(0,T;X)=\{f\in
\mathcal{M}(0,T;X)|\|f\|_{L^r_T(X)}<\infty\}$. If $T=\infty$, we
denote by $L^r(\R^+; X)$ and $\|f\|_{L^r(X)}$ the corresponding
spaces and norms, respectively. $C([0,T],X)$ (or $C(\R^+,X)$) stands for the
set of continuous X-valued functions on $[0,T]$ (resp. $\R^+$) while $C_b(\R^+;X)$ is the set of bounded
continuous X-valued functions. For $\alpha\in(0,1)$,
$C^\alpha([0,T];X)$ (or $C^\alpha(\R^+;X)$) stands for the set of
H\"{o}lder continuous functions in time with order $\alpha$,
i.e., for every $t,s$ in $[0,T]$ (resp. $\R^+$), we have
$$\|f(t)-f(s)\|_X\lesssim |t-s|^\alpha.$$

As in \cite{RD2}, we introduce a
function $\psi\in C^\infty(\R^N)$, supported in
$\mathcal{C}=\{\xi\in\R^N: \f{5}{6}\le|\xi|\le\f{12}{5}\}$ and
such that
$$\sum_{q\in\mathbb{Z}}\psi(2^{-q}\xi)=1\textrm{ if }\xi\neq 0.$$ Let
$\mathcal{F}$ be the Fourier transform. Denoting $\mathcal{F}^{-1}\psi$ by $h$, we define the dyadic
blocks as follows:
$$\D_q f=\psi(2^{-q}D)f=2^{qN}\int_{\R^N}h(2^qy)f(x-y)dy,\quad \text{and}\ \
S_q f=\sum_{p\le q-1}\D_pf.$$
Then the formal decomposition
\begin{equation}
f=\sum_{q\in\mathbb{Z}}\D_qf\nonumber
\end{equation}
is called homogeneous Littlewood--Paley decomposition.

For $s\in\R$ and $f\in \mathcal{S}'(\R^N)$, we denote
$$\|f\|_{\dot{B}^s_{p,r}}\overset{def}{=}\left(\sum_{q\in\mathbb{Z}}2^{sqr}\|\D_qf\|^r_{L^p}\right)^{\f{1}{r}}.$$
When $p=2$ and $r=1$, we denote $\|\cdot\|_{\dot{B}^s_{p,r}}$ by
$\|\cdot\|_{B^s}$. The definition of the homogeneous Besov space is built on the
homogeneous Littlewood--Paley decomposition (cf. \cite[Definition 1.2]{RD2})
\begin{Definition}
Let $s\in\R$, and $m=-\left[\f{N}{2}+1-s\right]$. If $m<0$, we set
$$B^s=\left\{f\in \mathcal{S}'(\R^N)|\|f\|_{B^s}<\infty\textrm{
and }f=\sum_{q\in\mathbb{Z}}\D_qf\textrm{ in
}\mathcal{S}'(\R^N)\right\}.$$ If $m\ge 0$, we denote by
$\mathcal{P}_m$ the set of $N$ variables polynomials of degree
$\le m$ and define
$$B^s=\left\{f\in \mathcal{S}'(\R^N)/\mathcal{P}_m|\|f\|_{B^s}<\infty\textrm{
and }f=\sum_{q\in\mathbb{Z}}\D_qf\textrm{ in
}\mathcal{S}'(\R^N)/\mathcal{P}_m\right\}.$$
\end{Definition}

Functions in the homogeneous Besov space $B^s$ has many good properties (see \cite[Proposition 2.5]{RD}):
\begin{Proposition}\label{p3}
The following properties hold:
\begin{itemize}
\item Density: the set $C_0^\infty$ is dense in $B^s$ if $|s|\le
\f{N}{2}$;
\item Derivation: $\|f\|_{B^s}\thickapprox \|\nabla f\|_{B^{s-1}}$;
\item Fractional derivation: let $\Lambda=\sqrt{-\D}$ and
$\sigma\in\R$; then the operator $\Lambda^\sigma$ is an isomorphism
from $B^s$ to $B^{s-\sigma}$;
\item Algebraic properties: for $s>0$, $B^s\cap L^\infty$ is an
algebra;
\item Interpolation: $(B^{s_1}, B^{s_2})_{\theta,1}=B^{\theta
s_1+(1-\theta)s_2}$, for $s_1, s_2\in \mathbb{R}$ and $\theta\in (0,1)$.
\end{itemize}
\end{Proposition}

To deal with functions with different regularities for high
frequencies and low frequencies, motivated by \cite{RD2}, it is
more effective to work in \textit{hybrid Besov spaces}. We
remark that using hybrid Besov spaces has been crucial for proving
global well-posedness for compressible Navier--Stokes equations in critical spaces
(see \cite{RD, RD2}).
\begin{Definition}
Let $s,t\in \R$. We set
$$\|f\|_{\tilde{B}^{s,t}}=\sum_{q\le
0}2^{qs}\|\D_qf\|_{L^2}+\sum_{q>0}2^{qt}\|\D_qf\|_{L^2}.$$
For $m=-\left[\f{N}{2}+1-s\right]$, we define
\begin{eqnarray}
\tilde{B}^{s,t}&=&\left\{f\in\mathcal{S}'(\R^N)|\|f\|_{\tilde{B}^{s,t}}<\infty\right\},\qquad \textrm{
if}\ m<0,\nonumber\\
\tilde{B}^{s,t}&=&\left\{f\in\mathcal{S}'(\R^N)/\mathcal{P}_m|\|f\|_{\tilde{B}^{s,t}}<\infty\right\},\ \textrm{
if }m\ge 0.\nonumber
\end{eqnarray}
\end{Definition}

\begin{Remark}
Some remarks about the hybrid Besov spaces are in order:
\begin{itemize}
\item $\tilde{B}^{s,s}=B^s$;
\item If $s\le t$, then $\tilde{B}^{s,t}=B^s\cap B^t$. Otherwise,
$\tilde{B}^{s,t}=B^s+B^t$. In particular,
$\tilde{B}^{s,\f{N}{2}}\hookrightarrow L^\infty$ as $s\le
\f{N}{2}$;
\item The space $\tilde{B}^{0,s}$ coincides with the usual
nonhomogeneous Besov space
$$\left\{f\in\mathcal{S}'(\R^N)|\|\chi(D)f\|_{L^2}+\sum_{q\ge
0}2^{qs}\|\D_qf\|_{L^2}<\infty\right\},$$ where
$\chi(\xi)=1-\sum_{q\ge 0}\phi(2^{-q}\xi)$;
\item If $s_1\le s_2$ and $t_1\ge t_2$, then
$\tilde{B}^{s_1,t_1}\hookrightarrow\tilde{B}^{s_2,t_2}$.
\end{itemize}
\end{Remark}
We have the following properties for the product in hybrid Besov spaces (see \cite{RD}):
\begin{Proposition}\label{p22}
 For all $s_1, s_2>0$,
$$\|fg\|_{\tilde{B}^{s_1,s_2}}\lesssim
\|f\|_{L^\infty}\|g\|_{\tilde{B}^{s_1,s_2}}+\|g\|_{L^\infty}\|f\|_{\tilde{B}^{s_1,s_2}}.$$
For all $s_1, s_2\le\f{N}{2}$ such that $\min\{s_1+t_1,
s_2+t_2\}>0$,
$$\|fg\|_{\tilde{B}^{s_1+s_2-\f{N}{2},t_1+t_2-\f{N}{2}}}\lesssim
\|f\|_{\tilde{B}^{s_1,s_2}}\|g\|_{\tilde{B}^{t_1,t_2}}.$$
\end{Proposition}

Throughout this paper, the following estimates for the convection
terms arising in the linearized systems will be used frequently (cf.
 \cite[Lemma 5.1]{RD2}).

\begin{Lemma}\label{cl}
Let $G$ be an homogeneous smooth function of degree $m$. Suppose
$-\f{N}{2}<s_i,t\le 1+\f{N}{2}$ for $i=1,2$. Then the following
three inequalities hold true:
\begin{equation}\label{25}
\begin{split}
&|(G(D)\D_q(\u\cdot\nabla f)|G(D)\D_q f)|\\
&\quad \leq C\alpha_q
2^{-q(\phi^{s_1,s_2}(q)-m)}
\|\u\|_{B^{1+\f{N}{2}}}\|f\|_{\tilde{B}^{s_1,s_2}}\left\|G(D)\D_q
f\right\|_{L^2},
\end{split}
\end{equation}
\begin{equation}\label{26}
\begin{split}
&\left|(G(D)\D_q(\u\cdot\nabla f)|\D_q g)+(\D_q(\u\cdot\nabla
g)|G(D)\D_q f)\right|\\
 &\quad\le
C\alpha_q\|\u\|_{B^{1+\f{N}{2}}}\Big(2^{-qt}\left\|G(D)\D_qf\right\|_{L^2}\|g\|_{B^t}+2^{-q(\phi^{s_1,s_2}(q)-m)}
\|f\|_{\tilde{B}^{s_1,s_2}}\|\D_qg\|_{L^2}\Big),
\end{split}
\end{equation}
where $\sum_{q\in\mathbb{Z}}\alpha_q\le 1$ and $C$ is a universal constant that only depends on $s_i, t, N$. The notation $\phi^{s,t}(q)$ means that for $q\in \mathbb{Z}$,
\begin{equation*}
\phi^{s,t}(q)\overset{def}=
\begin{cases}
s,\quad\textrm{if}\quad q\le 0,\\
t,\quad\textrm{if}\quad q\ge 1.
\end{cases}
\end{equation*}
\end{Lemma}
\begin{Remark} The above lemma looks slightly different from the original one \cite[Lemma 5.1]{RD2}. Indeed, \eqref{25} and \eqref{26} are corresponding to (41) and (42) in \cite{RD2}, respectively, up to a change with the notation $\phi^{s,t}$.
\end{Remark}

\bigskip
\section{Reformulation of the Original System \eqref{e1} and Main Result}
In this section, we first reformulate the original system
\eqref{e1} into a different form and then we state our main result
on the global existence of strong solutions. We simply set
$\theta=\xi=1$ since their sizes do not play any role in our analysis.
For $s\in\R$, we denote
$$\Lambda^sf:=\mathcal{F}^{-1}(|\xi|^s\mathcal{F}(f)).$$
Using the idea in \cite{RD2}, we decompose the velocity field into a compressible part and an incompressible part. Let
$$h=\Lambda^{-1}\Dv\u$$
and
$$\O=\Lambda^{-1}\textrm{curl}\u, \quad\text{with }
(\textrm{curl}\u)_i^j=\partial_{x_j}\u^i-\partial_{x_i}\u^j.$$
Owing to the identity
$\D=\nabla\Dv-\textrm{curl}\textrm{curl}$, we have the decomposition
$$\u=-\Lambda^{-1}\nabla h+\Lambda^{-1}\textrm{curl}\O,$$ which implies that $\u$ can be
recovered from the information of $h$ and $\O$. Denote
$$\mathcal{A}\overset{def}=\mu\D+(\lambda+\mu)\nabla\Dv,$$
and
\begin{equation}
\mathcal{N}=-\u\cdot\nabla\u-\f{\r-1}{\r}\mathcal{A}\u-\f{1}{\r}\Dv\left(\nabla\d\odot\nabla\d-\f12|\nabla\d|^2I\right).\label{N}
\end{equation}
Applying $\Lambda^{-1}\Dv$ and $\Lambda^{-1}\textrm{curl}$ to the
moment equation in \eqref{e1} respectively, we obtain that
\begin{equation}\label{21}
\begin{cases}
\partial_t h-\nu\D h=\Lambda^{-1}\Dv\left(\mathcal{N}-\frac{1}{\rho}\nabla P(\rho)\right),\\
\partial_t\O-\mu\D\O=\Lambda^{-1}\textrm{curl}\mathcal{N},
\end{cases}
\end{equation}
where $\nu=2\mu+\lambda$. Since $\mu>0$ and $2\mu+3\lambda\geq 0$, we have $\nu\geq \frac43\mu>0$. In the second equation of \eqref{21}, we have used the fact that
$$\textrm{curl}\left(\frac{1}{\rho}\nabla P(\rho)\right)= \frac{1}{\rho} \textrm{curl}\nabla P(\rho)+\nabla\left(\frac{1}{\rho}\right)\times \nabla P(\r)=0.$$
The advantage of the above reformulation is to get rid of the
pressure for $\O$, the incompressible part of the velocity, while
we still keep all information of the velocity field $\u$.

For the simplicity of our presentation, our proof focuses on the case:
$P(\r)=\f12\r^2$. The general barotropic case ($P(\r)$ is an increasing convex function of $\r$) can be verified by a slight modification of the argument below.

In this paper, we shall prove the existence of global strong solution for initial datum that is close to an equilibrium state $(1,0,
\hat{\d})$ with a constant vector $\hat{\d}\in S^2$. The result is valid for any positive constant density $\hat{\r}$ and we take $\hat{\r}=1$ just for simplicity. Keeping \eqref{21} in mind, it is convenient to reformulate the original system \eqref{e1} into a new system in terms of $\r$, $h$,
$\O$ and $\d$
\begin{subequations}\label{e2}
\begin{align}
&\partial_t(\r-1)+\u\cdot \nabla (\rho-1)+\Lambda h=-(\r-1)\Dv \u,\\
&\partial_t h-\nu\D h-\Lambda(\r-1)=\Lambda^{-1}\Dv\mathcal{N},\\
&\partial_t\O-\mu\D\O=\Lambda^{-1}\textrm{curl}\mathcal{N},\label{e2c}\\
&\f{\partial\d}{\partial t}+\u\cdot\nabla\d-\D\d=|\nabla\d|^2\d,
\end{align}
\end{subequations}
subject to initial conditions
\begin{equation}
\begin{split}
& \r|_{t=0}=\r_0(x),\quad h|_{t=0}=h_0(x)=\Lambda^{-1}\Dv\u_0(x),\\
&\ \O|_{t=0}=\O(x)=\Lambda^{-1}\textrm{curl}\u_0(x),\quad \d|_{t=0}=\d_0(x).
\end{split}
\end{equation}

Let us now introduce the functional space that appears in the
global existence theorem.
\begin{Definition}
For $T>0$, and $s\in \R$, we denote
\begin{equation*}
\begin{split}
\mathfrak{B}^s_T&=\Big\{(e, \mathbf{f}, \mathbf{g})\in  \left(L^1(0,T;
\tilde{B}^{s+1,s})\cap C([0,T];\tilde{B}^{s-1,s})\right)\\
 &\qquad\qquad\qquad\times\left( L^1(0,T; B^{s+1})\cap
C([0,T];B^{s-1})\right)^3\\
&\qquad\qquad\qquad\times \left(L^1(0,T;
\tilde{B}^{s+1,s+2})\cap C([0,T];\tilde{B}^{s-1,s})\right)^3\Big\}
\end{split}
\end{equation*}
and
\begin{equation*}
\begin{split}
\|(e,\mathbf{f}, \mathbf{g})\|_{\mathfrak{B}^s_T}&=\|e\|_{L^\infty_T(\tilde{B}^{s-1,s})}+\|\mathbf{f}\|_{L^\infty_T(B^{s-1})}+\|\mathbf{g}\|_{L^\infty_T(\tilde{B}^{s-1,s})}\\
&\quad+\|e\|_{L^1_T(\tilde{B}^{s+1,s})}+\|\mathbf{f}\|_{L^1_T(B^{s+1})}+\|\mathbf{g}\|_{L^1_T(\tilde{B}^{s+1,s+2})}.
\end{split}
\end{equation*}
We use the notation $\mathfrak{B}^s$ if $T=+\infty$ by changing
the interval $[0,T]$ into $[0,\infty)$ in the definition above.
\end{Definition}

The main result of this paper is as follows:
\begin{Theorem}\label{MT} Let $\hat{\d}\in\R^3$ be an arbitrary constant unit
vector. There exists two positive constants $\eta$ and $\Gamma$,
such that, if $\r_0-1\in\tilde{B}^{\f{1}{2},\f{3}{2}}$, $\u_0\in
B^{\f{1}{2}}$ and $\d_0-\hat{\d}\in\tilde{B}^{\f{1}{2},\f{3}{2}}$
satisfy
$$\|\r_0-1\|_{\tilde{B}^{\f{1}{2},\f{3}{2}}}+\|\u_0\|_{B^{\f{1}{2}}}+\|\d_0-\hat{\d}\|_{\tilde{B}^{\f{1}{2},\f{3}{2}}}\le
\eta,$$ then system \eqref{e1}--\eqref{ic} has a unique global strong solution
$(\r, \u, \d)$ with $(\r-1, \u, \d-\hat{\d})$ in
$\mathfrak{B}^{\f{3}{2}}$ satisfying
$$\|(\r-1,\u,\d-\hat{\d})\|_{\mathfrak{B}^{\f{3}{2}}}\le
\Gamma\left(\|\r_0-1\|_{\tilde{B}^{\f{1}{2},\f{3}{2}}}+\|\u_0\|_{B^{\f{1}{2}}}+\|\d_0-\hat{\d}\|_{\tilde{B}^{\f{1}{2},\f{3}{2}}}\right).$$
\end{Theorem}

The local existence in $\mathfrak{B}^{\f{3}{2}}_T$ can be
established by a standard fixed point argument, for instance, see
\cite{RD2}, and in particular the unique solution satisfies $\d(t,x)\in
S^2$ whenever it exists. The global existence of \eqref{e1}--\eqref{ic} will be
established by extending a local solution with the help of uniform
estimates for the local solution when the initial data is
sufficiently ``small". In the rest of this paper, we focus on the
uniform estimates and uniqueness of the solution to
\eqref{e1}--\eqref{ic}.

\begin{Remark}
Similar results for other dimensions are still true. As the
density is equal to one, we recover a global existence result for incompressible liquid crystal flows, which is
similar to \cite{WC}.
\end{Remark}

\bigskip

\section{Uniform Estimates for Linearized Systems with Convection}

In this section, our goal is to obtain uniform estimates of local
solutions. For this purpose, we consider proper linearized systems with convection that are associated with the reformulated system \eqref{e2}. First, we investigate the following linearized equations for $\O$ and $\d$.
\begin{subequations}\label{le}
\begin{align}
&\partial_t\O+\u\cdot\nabla\O-\mu\D\O
       =\mathcal{L},\label{le1}\\
&\partial_t\d+\u\cdot\nabla\d-\D\d=\mathcal{M},\label{le2}
\end{align}
\end{subequations}
where $\mathcal{L}$, $\mathcal{M}$ and $\u$ are given functions.
For the system \eqref{le}, we have the following estimate:

\begin{Proposition}\label{p1}
Denote
\begin{equation}
V(t):=\int_0^t\|\u(s)\|_{B^{\f{5}{2}}}ds. \label{VV}
\end{equation}
Let $(\O, \d)$ be a solution of \eqref{le} on $[0,T)$. Then the following estimate holds  on $[0,T)$:
\begin{equation*}
\begin{split}
&\|\d(t)\|_{\tilde{B}^{\f{1}{2},\f{3}{2}}}
+\|\O(t)\|_{B^{\f{1}{2}}}+\int_0^t\left(\|\d(s)\|_{\tilde{B}^{\f{5}{2},\f{7}{2}}}+\|\O(s)\|_{B^{\f{5}{2}}}\right)ds\\&\quad\le
Ce^{CV(t)}\Big\{\|\d_0\|_{\tilde{B}^{\f{1}{2},\f{3}{2}}}
+\|\O_0\|_{B^{\f{1}{2}}}+\int_0^te^{-CV(s)}\Big(\|\mathcal{L}\|_{B^{\f{1}{2}}}+\|\mathcal{M}\|_{\tilde{B}^{\f{1}{2},\f{3}{2}}}\Big)ds\Big\},
\end{split}
\end{equation*}
where $C$ is a universal positive constant.
\end{Proposition}
\begin{proof}
To prove this proposition, we first localize \eqref{le} into low and
high frequencies according to the Littlewood--Paley
decomposition. Then each
dyadic block can be estimated by using energy method.

Let $(\O, \d)$ be a solution of \eqref{le} and $K>0$. We introduce the following transformation of variables \cite{RD2}:
$$\tilde{\O}=e^{-KV(t)}\O,
\quad\tilde{\d}=e^{-KV(t)}\d,\quad \tilde{\mathcal{L}}=e^{-KV(t)}\mathcal{L},\quad
\tilde{\mathcal{M}}=e^{-KV(t)}\mathcal{M}.$$
Applying the operator
$\D_q$ to the system \eqref{le}, we deduce that $(\D_q\tilde{\O},
\D_q\tilde{\d})$ satisfies
\begin{equation}\label{32}
\begin{cases}
\partial_t\D_q\tilde{\O}+\D_q(\u\cdot\nabla\tilde{\O})-\mu\D\D_q\tilde{\O}=\D_q\tilde{\mathcal{L}}
-KV'(t)\D_q\tilde{\O},\\
\partial_t \D_q\tilde{\d}+\D_q(\u\cdot\nabla\tilde{\d})-\D
\D_q\tilde{\d}=\D_q\tilde{\mathcal{M}}-KV'(t)\D_q\tilde{\d}.
\end{cases}
\end{equation}
The proof can be carried out in three
steps.

{\bf Step 1: Low Frequencies.}\quad Suppose $q\le 0$ and
define
\begin{equation*}
\begin{split}
f_q^2&=\|\D_q\tilde{\O}\|^2_{L^2}+\|\D_q\tilde{\d}\|^2_{L^2}.
\end{split}
\end{equation*}
Taking the $L^2$-scalar product of the first equation of
\eqref{32} with $\D_q\tilde{\O}$, and the second equation with
$\D_q\tilde{\d}$, we obtain the following two identities:
\begin{equation}\label{33}
\begin{split}
\f{1}{2}\f{d}{dt}\|\D_q\tilde{\O}\|_{L^2}+(\D_q(\u\cdot\nabla\tilde{\O})|\D_q\tilde{\O})+\mu\|\Lambda\D_q\tilde{\O}\|^2_{L^2}=(\D_q\tilde{\mathcal{L}}|\D_q\tilde{\O})-KV'\|\D_q\tilde{\O}\|_{L^2}^2,
\end{split}
\end{equation}
and
\begin{equation}\label{34}
\begin{split}
\f{1}{2}\f{d}{dt}\|\D_q\tilde{\d}\|_{L^2}^2+(\D_q(\u\cdot\nabla\tilde{\d})|\D_q\tilde{\d})+\|\Lambda\D_q\tilde{\d}\|_{L^2}^2
=(\D_q\tilde{\mathcal{M}}|\D_q\tilde{\d})-KV'\|\D_q\tilde{\d}\|_{L^2}^2.
\end{split}
\end{equation}
Adding \eqref{33} and \eqref{34} together, we obtain
\begin{equation}\label{35}
\f{1}{2}\f{d}{dt}f^2_q+\mu\|\Lambda\D_q\tilde{\O}\|^2_{L^2}+\|\Lambda\D_q\tilde{\d}\|_{L^2}^2=\mathcal{X}-KV'(t)f_q^2,
\end{equation}
where
\begin{equation*}
\mathcal{X}=-(\D_q(\u\cdot\nabla\tilde{\O})|\D_q\tilde{\O})-(\D_q(\u\cdot\nabla\tilde{\d})|\D_q\tilde{\d})+(\D_q\tilde{\mathcal{L}}|\D_q\tilde{\O})+(\D_q\tilde{\mathcal{M}}|\D_q\tilde{\d}).
\end{equation*}
The term $\mathcal{X}$ can be estimated by using Lemma \ref{cl} (taking $s_1=s_2=\frac12$ in \eqref{25}) such that
\begin{equation}\label{36}
\begin{split}
|\mathcal{X}|&\le
Cf_q\Big(\|\D_q\tilde{\mathcal{L}}\|_{L^2}+\|\D_q\tilde{\mathcal{M}}\|_{L^2}+2^{-\f{q}{2}}\alpha_q
V'(\|\tilde{\d}\|_{B^{\f{1}{2}}}+\|\tilde{\O}\|_{B^{\f{1}{2}}})\Big)\\
&\le
Cf_q\Big(\|\D_q\tilde{\mathcal{L}}\|_{L^2}+\|\D_q\tilde{\mathcal{M}}\|_{L^2}+2^{-\f{q}{2}}\alpha_q
V'(\|\tilde{\d}\|_{\tilde{B}^{\f{1}{2},\f{3}{2}}}+\|\tilde{\O}\|_{B^{\f{1}{2}}})\Big).
\end{split}
\end{equation}
 In the last inequality, we used the embedding
$$\tilde{B}^{\f{1}{2},\f{3}{2}}\hookrightarrow B^{\f{1}{2}}.$$
Hence, combining \eqref{35} and \eqref{36} together, we have for $q\leq 0$
\begin{equation}\label{37}
\begin{split}
&\f{1}{2}\f{d}{dt}f_q^2+\mu\|\Lambda\D_q\tilde{\O}\|^2_{L^2}+\|\Lambda\D_q\tilde{\d}\|_{L^2}^2\\
&\le
Cf_q\Big(\|\D_q\tilde{\mathcal{L}}\|_{L^2}+\|\D_q\tilde{\mathcal{M}}\|_{L^2}+2^{-\f{q}{2}}\alpha_q
V'(\|\tilde{\d}\|_{\tilde{B}^{\f{1}{2},\f{3}{2}}}+\|\tilde{\O}\|_{B^{\f{1}{2}}})\Big)-KV'(t)f_q^2.
\end{split}
\end{equation}

{\bf Step 2: High Frequencies.}\quad In this step, we assume
$q>0$ and set
$$f_q^2=\|\D_q\tilde{\O}\|^2_{L^2}+\|\Lambda\D_q\tilde{\d}\|^2_{L^2}$$
We apply the operator $\Lambda$ to the second equation of
\eqref{32}, multiply by $\Lambda\D_q\tilde{\d}$ and integrate over
$\R^3$ to yield
\begin{equation}\label{38}
\begin{split}
&\f{1}{2}\f{d}{dt}\|\Lambda\D_q\tilde{\d}\|_{L^2}^2+(\Lambda\D_q(\u\cdot\nabla\tilde{\d})|\Lambda\D_q{\tilde{\d}})
+\|\Lambda^2\D_q\tilde{\d}\|_{L^2}^2\\
& \quad =(\Lambda\D_q\tilde{\mathcal{M}}|\Lambda\D_q\tilde{\d})-KV'\|\Lambda\D_q\tilde{\d}\|^2_{L^2}.
\end{split}
\end{equation}
Adding \eqref{33} and \eqref{38} together, we have
\begin{equation}\label{39}
\f{1}{2}\f{d}{dt}f_q^2+\mu\|\Lambda\D_q\tilde{\O}\|^2_{L^2}+\|\Lambda^2\D_q\tilde{\d}\|_{L^2}^2=\mathcal{Y}-KV'(t)f_q^2
\end{equation}
with
\begin{equation*}
\mathcal{Y}=-(\D_q(\u\cdot\nabla\tilde{\O})|\D_q\tilde{\O})-(\Lambda\D_q(\u\cdot\nabla\tilde{\d})|\Lambda\D_q{\tilde{\d}})+(\D_q\tilde{\mathcal{L}}|\D_q\tilde{\O})+(\Lambda\D_q\tilde{\mathcal{M}}|\Lambda\D_q\tilde{\d}).
\end{equation*}
For $\mathcal{Y}$, using Lemma \ref{cl}, one has
\begin{equation}\label{310}
\begin{split}
|\mathcal{Y}| &\le
 \|\D_q\tilde{\O}\|_{L^2}\|\D_q\tilde{\mathcal{L}}\|_{L^2}+\|\Lambda\D_q\tilde{\d}\|_{L^2}\|\Lambda\D_q\tilde{\mathcal{M}}\|_{L^2}
 +C2^{-\f{q}{2}}\alpha_q
V'\|\tilde{\O}\|_{B^{\f{1}{2}}}\|\D_q\tilde{\O}\|_{L^2}\\
&\quad + C2^{-q(\phi^{\frac12, \frac32}(q)-1)}\alpha_q
V'  \|\tilde{\d}\|_{\tilde{B}^{\f{1}{2},\f{3}{2}}}\|\Lambda\D_q\tilde{\d}\|_{L^2}
\\
&\leq Cf_q\Big(\|\D_q\tilde{\mathcal{L}}\|_{L^2}+\|\Lambda\D_q\tilde{\mathcal{M}}\|_{L^2}+2^{-\f{q}{2}}\alpha_q
V'(\|\tilde{\d}\|_{\tilde{B}^{\f{1}{2},\f{3}{2}}}+\|\tilde{\O}\|_{B^{\f{1}{2}}})\Big)
\end{split}
\end{equation}
 Hence, combining \eqref{39} and \eqref{310}
together, we have for $q>0$
\begin{equation}\label{311}
\begin{split}
&\f{1}{2}\f{d}{dt}f_q^2+\mu\|\Lambda\D_q\tilde{\O}\|^2_{L^2}+\|\Lambda^2\D_q\tilde{\d}\|_{L^2}^2\\
&\ \le
Cf_q\Big(\|\D_q\tilde{\mathcal{L}}\|_{L^2}+\|\Lambda\D_q\tilde{\mathcal{M}}\|_{L^2}+2^{-\f{q}{2}}\alpha_q
V'(\|\tilde{\d}\|_{\tilde{B}^{\f{1}{2},\f{3}{2}}}+\|\tilde{\O}\|_{B^{\f{1}{2}}})\Big)-KV'(t)f_q^2.
\end{split}
\end{equation}

{\bf Step 3: Damping Effect.}\quad We are now going to show
that inequalities \eqref{37} and \eqref{311} entail a decay for
$\O$ and $\d$. Denote $g_q=2^{\f{q}{2}}f_q$ for $q\in \mathbb{Z}$. The well-known Bernstein's inequality implies that
\begin{equation}
 C_0g_q\le
2^{\f{q}{2}}\|\D_q\tilde{\O}\|_{L^2}+2^{q\phi^{\f{1}{2},\f{3}{2}}(q)}\|\D_q\tilde{\d}\|_{L^2}\le\f{1}{C_0}g_q\label{BBg}
 \end{equation}
for some universal positive constant $C_0$.
Therefore, we infer from \eqref{37}, \eqref{311} that there exists a universal positive constant $\kappa$ such that
\begin{equation}\label{312}
\begin{split}
\f{1}{2}\f{d}{dt}g_q^2+\kappa \min\{\mu,1\}2^{2q}g_q^2&\le C\alpha_q
g_q\Big(\|\tilde{\mathcal{L}}\|_{B^{\f{1}{2}}}+\|\tilde{\mathcal{M}}\|_{\tilde{B}^{\f{1}{2},\f{3}{2}}}+
V'(\|\tilde{\d}\|_{\tilde{B}^{\f{1}{2},\f{3}{2}}}+\|\tilde{\O}\|_{B^{\f{1}{2}}})\Big)\\
&\quad-KV'(t)g_q^2.
\end{split}
\end{equation}
Let $\dl>0$ be a small parameter and denote
$\chi_q^2=g_q^2+\dl^2.$ From \eqref{312}, dividing by
$\chi_q$, we obtain
\begin{equation*}
\begin{split}
\f{d}{dt}\chi_q+\kappa \min\{\mu,1\}2^{2q} \chi_q &\le
C\alpha_q\Big(\|\tilde{\mathcal{L}}\|_{B^{\f{1}{2}}}+\|\tilde{\mathcal{M}}\|_{\tilde{B}^{\f{1}{2},\f{3}{2}}}+
V'(\|\tilde{\d}\|_{\tilde{B}^{\f{1}{2},\f{3}{2}}}+\|\tilde{\O}\|_{B^{\f{1}{2}}})\Big)\\&\quad-KV'\chi_q+\dl
KV'+\dl \kappa 2^{2q}.
\end{split}
\end{equation*}
Integrating the above inequality over $[0,t]$ and having $\dl$
tend to $0$, we obtain,
\begin{equation}\label{313}
\begin{split}
&g_q(t)+\kappa \min\{\mu,1\}2^{2q}\int_0^t g_q(s)ds\\
&\ \le
g_q(0)+C\int_0^t\alpha_q(s)\Big(\|\tilde{\mathcal{L}}\|_{B^{\f{1}{2}}}+\|\tilde{\mathcal{M}}\|_{\tilde{B}^{\f{1}{2},\f{3}{2}}}\Big)ds\\
&\quad+\int_0^tV'(s)\Big(C\alpha_q(s)
(\|\tilde{\d}\|_{\tilde{B}^{\f{1}{2},\f{3}{2}}}+\|\tilde{\O}\|_{B^{\f{1}{2}}})-Kg_q(s)\Big)ds.
\end{split}
\end{equation}
\eqref{BBg} implies that
\begin{equation*}
\begin{split}
&C\alpha_q(s)
(\|\tilde{\d}\|_{\tilde{B}^{\f{1}{2},\f{3}{2}}}+\|\tilde{\O}\|_{B^{\f{1}{2}}})-Kg_q(s)\\
&\ \ \le
C\alpha_q(s)\|\tilde{\d}\|_{\tilde{B}^{\f{1}{2},\f{3}{2}}}-C_0K2^{q\phi^{\f{1}{2},\f{3}{2}}(q)}\|\D_q\tilde{\d}\|_{L^2}+C\alpha_q(s)\|\tilde{\O}\|_{B^{\f{1}{2}}}-C_0K
2^{\f{q}{2}}\|\D_q\tilde{\O}\|_{L^2}.
\end{split}
\end{equation*}
If we choose $K$ such that $KC_0>C$, then we infer from the fact $\sum_{q\in\mathbb{Z}}\alpha_q\le 1$ that
\begin{equation}\label{314}
\sum_{q\in\mathbb{Z}}\Big\{C\alpha_q(s)
\left(\|\tilde{\d}\|_{\tilde{B}^{\f{1}{2},\f{3}{2}}}+\|\tilde{\O}\|_{B^{\f{1}{2}}}\right)-Kg_q(s)\Big\}\le
0.
\end{equation}
With the inequality \eqref{314} in hand, after summation over
$\mathbb{Z}$, we deduce from \eqref{313} that
\begin{equation}\label{315}
\begin{split}
&\|\tilde{\O}(t)\|_{B^{\f{1}{2}}}+\|\tilde{\d}(t)\|_{\tilde{B}^{\f{1}{2},\f{3}{2}}}+\kappa\min\{\mu,1\}\int_0^t\left(\|\tilde{\O}(s)\|_{\tilde{B}^{\f{5}{2}}}+\|\tilde{\d}(s)\|_{\tilde{B}^{\f{5}{2},\f{7}{2}}}
\right)ds
\\&\quad\le C\Big\{\|\O_0\|_{B^{\f{1}{2}}}+\|\d_0\|_{\tilde{B}^{\f{1}{2},\f{3}{2}}}
+\int_0^t\Big(\|\tilde{\mathcal{L}}(s)\|_{\tilde{B}^{\f{1}{2}}}+\|\tilde{\mathcal{M}}(s)\|_{\tilde{B}^{\f{1}{2},\f{3}{2}}}\Big)ds\Big\}.
\end{split}
\end{equation}
This finishes the proof.
\end{proof}

Next, we turn to consider the linearized system for $\vr$ and $h$:
\begin{subequations}\label{2e}
\begin{align}
&\partial_t \vr+\u\cdot\nabla \vr+\Lambda h
       =\mathcal{J},\label{le1b}\\
&\partial_t h+\u\cdot\nabla h-\nu\D
h-\Lambda\vr=\mathcal{K},\label{le2b}
\end{align}
\end{subequations}
where $\mathcal{J}$, $\mathcal{K}$ and $\u$ are given functions.
Note that this system is different from \eqref{le} since there are
stronger couplings between $\vr$ and $h$. For the system \eqref{2e},
we have the following estimates (see a similar result in \cite[Proposition 2.3]{RD2}):

\begin{Proposition}\label{p2} Let $(\vr, h)$ be a solution of
\eqref{2e} on $[0,T)$. Then the following estimate holds  on
$[0,T)$:
\begin{equation*}
\begin{split}
&\|\vr(t)\|_{\tilde{B}^{\f{1}{2},\f{3}{2}}}
+\|h(t)\|_{B^{\f{1}{2}}}+\int_0^t\left(\|\vr(s)\|_{\tilde{B}^{\f{5}{2},\f{3}{2}}}+\|h(s)\|_{B^{\f{5}{2}}}\right)ds\\&\quad\le
Ce^{CV(t)}\Big\{\|\vr_0\|_{\tilde{B}^{\f{1}{2},\f{3}{2}}}
+\|h_0\|_{B^{\f{1}{2}}}+\int_0^te^{-CV(s)}\Big(\|\mathcal{K}\|_{B^{\f{1}{2}}}+\|\mathcal{J}\|_{\tilde{B}^{\f{1}{2},\f{3}{2}}}\Big)ds\Big\},
\end{split}
\end{equation*}
where $C$ is a universal positive constant and $V$ is given by \eqref{VV}.
\end{Proposition}
\begin{proof}
The proof is due to  the argument in \cite{RD2}, and similar to
that in Proposition \ref{p1}. For the completeness, below we present a proof that is
slightly different from \cite{RD2}. To this end, again
we first localize \eqref{2e} in low and high frequencies according
to the Littlewood--Paley decomposition.

Let $(\vr, h)$ be a solution of \eqref{2e} and $K>0$. Define
$$\tilde{\vr}=e^{-KV(t)}\vr,\quad\tilde{h}=e^{-KV(t)}h,
\quad \tilde{\mathcal{J}}=e^{-KV(t)}\mathcal{J},\quad
\tilde{\mathcal{K}}=e^{-KV(t)}\mathcal{K}.$$ Applying the operator
$\D_q$ to the system \eqref{2e}, we deduce that $(\D_q\tilde{h},
\D_q\tilde{\r})$ satisfies
\begin{equation}\label{132}
\begin{cases}
\partial_t \D_q\tilde{\vr}+\D_q(\u\cdot\nabla\tilde{\vr})+\Lambda
\D_q\tilde{h}=\D_q\tilde{\mathcal{J}}-KV'(t)\D_q\tilde{\vr},\\
\partial_t\D_q\tilde{h}+\D_q(\u\cdot\nabla\tilde{h})-\nu\D\D_q\tilde{h}-\Lambda\D_q\tilde{\vr}=\D_q\tilde{\mathcal{K}}
-KV'(t)\D_q\tilde{h}.
\end{cases}
\end{equation}
Set
\begin{equation}
q_0=\log_2\left(\f{3}{\nu}\right).\label{q0}
\end{equation}

{\bf Step 1: Low Frequencies.}\quad Suppose $q\le q_0$. As a result, $2^q\leq 2^{q_0}\leq \f{3}{\nu}$. Taking the $L^2$-scalar
product of the first equation of \eqref{132} with $\D_q\tilde{\vr}$ and the second equation of \eqref{132} with $\D_q\tilde{h}$, we obtain the
following two identities:
\begin{equation}\label{134}
\begin{split}
\f{1}{2}\f{d}{dt}\|\D_q\tilde{\vr}\|_{L^2}^2+(\D_q(\u\cdot\nabla\tilde{\vr})|\D_q\tilde{\vr})+(\Lambda\D_q\tilde{h}|\D_q\tilde{\vr})
=(\D_q\tilde{\mathcal{J}}|\D_q\tilde{\vr})-KV'\|\D_q\tilde{\vr}\|_{L^2}^2,
\end{split}
\end{equation}
and
\begin{equation}\label{133}
\begin{split}
&\f{1}{2}\f{d}{dt}\|\D_q\tilde{h}\|^2_{L^2}+(\D_q(\u\cdot\nabla\tilde{h})|\D_q\tilde{h})+\nu\|\Lambda\D_q\tilde{h}\|^2_{L^2}-(\Lambda\D_q\tilde{\vr}|\D_q\tilde{h})\\
&\quad=(\D_q\tilde{\mathcal{K}}|\D_q\tilde{h})-KV'\|\D_q\tilde{h}\|_{L^2}^2.
\end{split}
\end{equation}
Next, we derive an identity involving
$(\Lambda\D_q\widetilde{\varrho}|\D_q\tilde{h})$. For this
purpose, we apply $\Lambda$ to the first equation in \eqref{132}
and take the $L^2$ scalar product with $\D_q\tilde{h}$, then take
the scalar product of the second equation in \eqref{132} with
$\Lambda\D_q\tilde{\vr}$. Summing up both equalities, we get
\begin{equation}\label{113}
\begin{split}
&\f{d}{dt}(\Lambda\D_q\widetilde{\varrho}|\D_q\tilde{h})+(\D_q(\u\cdot\nabla\tilde{h})|\Lambda\D_q\tilde{\varrho})+(\Lambda\D_q(\u\cdot\nabla\tilde{\varrho})|\D_q\tilde{h})\\
&\qquad-\|\Lambda\D_q\widetilde{\varrho}\|_{L^2}^2+\|\Lambda\D_q\tilde{h}\|_{L^2}^2+\nu(\Lambda^2\D_q\tilde{h}|\Lambda\D_q\tilde{\varrho})\\
&\ \
=(\Lambda\D_q\tilde{\mathcal{J}}|\D_q\tilde{h})+(\D_q\tilde{\mathcal{K}}|\Lambda\D_q\tilde{\varrho})-2KV'(\Lambda\D_q\tilde{\varrho}|\D_q\tilde{h})
\end{split}
\end{equation}
Let $\tau$ be a small constant such that $0<\tau\leq\frac{8}{9}$. We define
\begin{equation*}
\begin{split}
f_q^2&=\|\D_q\tilde{\vr}\|^2_{L^2}+\|\D_q\tilde{h}\|^2_{L^2}-\f{\tau\nu}{4}(\Lambda\D_q\tilde{\vr}|\D_q\tilde{h}).
\end{split}
\end{equation*}
The Bernstein's inequality yields
\begin{equation*}
 \begin{split}
  \left|\f{\tau\nu}{4}(\Lambda\D_q\tilde{\vr}|\D_q\tilde{h})\right|&\le\f{\tau\nu}{4}\|\Lambda\D_q\tilde{\vr}\|_{L^2}\|\D_q\tilde{h}\|_{L^2}\le \f{\tau\nu}{4}2^q\|\D_q\tilde{\vr}\|_{L^2}\|\D_q\tilde{h}\|_{L^2}\\
&\le 2^{q_0}\f{\tau\nu}{8}(\|\D_q\tilde{\vr}\|^2_{L^2}+\|\D_q\tilde{h}\|_{L^2}^2)\\
&\le \f{1}{3}(\|\D_q\tilde{\vr}\|^2_{L^2}+\|\D_q\tilde{h}\|_{L^2}^2)
 \end{split}
\end{equation*}
which implies that
\begin{equation}
 f_q^2\approx \|\D_q\tilde{h}\|^2_{L^2}+\|\D_q\tilde{\vr}\|^2_{L^2}.\label{fqq1}
\end{equation}
Here we note that the universal constant due to the Bernstein's inequality is harmless in our estimate and thus assumed to be one for simplicity.

Multiplying \eqref{113} by $-\frac{\tau\nu}{8}$ and adding it with \eqref{134}, \eqref{133}, we obtain that
\begin{equation}\label{135}
\f{1}{2}\f{d}{dt}f^2_q+\f{\nu}{8}\Big[(8-\tau)\|\Lambda\D_q\tilde{h}\|^2_{L^2}+\tau\|\Lambda\D_q\tilde{\vr}\|_{L^2}^2-\nu\tau(\Lambda^2\D_q\tilde{h}|\Lambda\D_q\tilde{\vr})\Big]=\mathcal{X}_1-KV'(t)f_q^2,
\end{equation}
with
\begin{equation*}
\begin{split}
\mathcal{X}_1&=-(\D_q(\u\cdot\nabla\tilde{h})|\D_q\tilde{h})-(\D_q(\u\cdot\nabla\tilde{\vr})|\D_q\tilde{\vr})+\f{\tau\nu}{8}(\D_q(\u\cdot\nabla\tilde{h})|\Lambda\D_q\tilde{\vr})\\
&\quad+\f{\tau\nu}{8}(\Lambda\D_q(\u\cdot\nabla\tilde{\vr})|\D_q\tilde{h})+(\D_q\tilde{\mathcal{K}}|\D_q\tilde{h})+(\D_q\tilde{\mathcal{J}}|\D_q\tilde{\vr})\\
&\quad-\frac{\tau\nu}{8}\left[(\Lambda\D_q\tilde{\mathcal{J}}|\D_q\tilde{h})+(\D_q\tilde{\mathcal{K}}|\Lambda\D_q\tilde{\vr})\right].
\end{split}
\end{equation*}
Since $\phi^{\frac12, \frac32}(q)\leq \frac32$, the assumption $q\leq q_0$ implies that $2^{-q(\phi^{\frac12, \frac32}(q)-\frac32)}\leq 2^{q_0}$. As a consequence, for $\mathcal{X}_1$, using Lemma \ref{cl} and \eqref{fqq1}, we have
\begin{equation}\label{136}
\begin{split}
|\mathcal{X}_1|&\leq C\alpha_q 2^{-\frac{q}{2}}V'(\|\tilde{h}\|_{B^\frac12}\|\D_q \tilde{h}\|_{L^2}+\|\tilde{\vr}\|_{B^\frac12}\|\D_q \tilde{\vr}\|_{L^2})
\\
& \ \ +\frac{C\tau\nu}{8} \alpha_q V'(2^{-\frac{q}{2}}\|\Lambda \Delta_q \tilde{\vr}\|_{L^2}\|\tilde{h}\|_{B^\frac12}
 +2^{-q(\phi^{\frac12, \frac32}(q)-1)}\|\tilde{\vr}\|_{\tilde{B}^{\f{1}{2},\f{3}{2}}}\|\Delta_q \tilde{h}\|_{L^2})
\\
&\ \ +\|\D_q\tilde{\mathcal{K}}\|_{L^2}\|\D_q \tilde{h}\|_{L^2}+\|\D_q\tilde{\mathcal{J}}\|_{L^2}\|\D_q \tilde{\vr}\|_{L^2}
\\
&\ \ + \frac{C\tau\nu}{8} 2^{q_0}(\|\D_q\tilde{\mathcal{K}}\|_{L^2}\|\D_q \tilde{\vr}\|_{L^2}+\|\D_q\tilde{\mathcal{J}}\|_{L^2}\|\D_q \tilde{h}\|_{L^2})\\
&\le
Cf_q\Big(\|\D_q\tilde{\mathcal{K}}\|_{L^2}+\|\D_q\tilde{\mathcal{J}}\|_{L^2}+2^{-\f{q}{2}}\alpha_q
V'(\|\tilde{\vr}\|_{\tilde{B}^{\f{1}{2},\f{3}{2}}}+\|\tilde{h}\|_{B^{\f{1}{2}}})\Big),
\end{split}
\end{equation}
where $C$ is a universal constant that may depend on $\nu$.
Besides, due to our choice of $\tau$, we can conclude that
\begin{equation}
|\nu\tau(\Lambda^2\D_q\tilde{h}|\Lambda\D_q\tilde{\vr})|\leq \nu\tau 2^{q_0}\|\Lambda\D_q\tilde{h}\|_{L^2}\|\Lambda\D_q\tilde{\vr}\|_{L^2}\leq 4\|\Lambda\D_q\tilde{h}\|^2_{L^2}+\frac{\tau}{2}\|\Lambda\D_q\tilde{\vr}\|_{L^2}^2.
\end{equation}
Then it easily follows from the Bernstein's inequality that
$$2^{2q}\Big(\|\D_q\tilde{h}\|^2_{L^2}+\|\D_q\tilde{\vr}\|_{L^2}^2\Big)\thickapprox(8-\tau)\|\Lambda\D_q\tilde{h}\|^2_{L^2}+\tau\|\Lambda\D_q\tilde{\vr}\|_{L^2}^2-\nu\tau(\Lambda^2\D_q\tilde{h}|\Lambda\D_q\tilde{\vr}).$$
Hence, combining \eqref{135} and \eqref{136}
together, we can find a positive universal constant $\kappa$ such that
\begin{equation}\label{137}
\begin{split}
&\f{1}{2}\f{d}{dt}f_q^2+\kappa
2^{2q}\Big(\|\D_q\tilde{h}\|^2_{L^2}+\|\D_q\tilde{\vr}\|_{L^2}^2\Big)\\
&\ \ \le
Cf_q\Big(\|\D_q\tilde{\mathcal{K}}\|_{L^2}+\|\D_q\tilde{\mathcal{J}}\|_{L^2}+2^{-\f{q}{2}}\alpha_q
V'(\|\tilde{\vr}\|_{\tilde{B}^{\f{1}{2},\f{3}{2}}}+\|\tilde{h}\|_{B^{\f{1}{2}}})\Big)-KV'(t)f_q^2.
\end{split}
\end{equation}

{\bf Step 2: High Frequencies.}\quad Suppose
$q\ge q_0$. We apply the operator
$\Lambda$ to the first equation of \eqref{132}, multiply by
$\Lambda\D_q\tilde{\vr}$ and integrate over $\R^3$ to yield
\begin{equation}\label{138}
\begin{split}
&\f{1}{2}\f{d}{dt}\|\Lambda\D_q\tilde{\vr}\|_{L^2}^2+(\Lambda\D_q(\u\cdot\nabla\tilde{\vr})|\Lambda\D_q{\tilde{\vr}})
+(\Lambda^2\D_q\tilde{h}|\Lambda\D_q\tilde{\vr})\\
&\ \ =(\Lambda\D_q\tilde{\mathcal{J}}|\Lambda\D_q\tilde{\vr})-KV'\|\Lambda\D_q\tilde{\vr}\|^2_{L^2}.
\end{split}
\end{equation}
 Set
$$f_q^2=\|\Lambda\D_q\tilde{\vr}\|^2_{L^2}+\f{3}{\nu^2}\|\D_q\tilde{h}\|^2_{L^2}-\f{2}{\nu}(\Lambda\D_q\tilde{\vr}\,|\D_q\tilde{h}).$$
It easily follow from the Cauchy--Schwarz inequality that
$$f_q^2\thickapprox\|\D_q\tilde{h}\|^2_{L^2}+\|\Lambda\D_q\tilde{\vr}\|^2_{L^2}.$$
A linear combination of \eqref{133}, \eqref{113} and \eqref{138}  yields that
\begin{equation}\label{139}
\f{1}{2}\f{d}{dt}f_q^2+\f{2}{\nu}\|\Lambda\D_q\tilde{h}\|^2_{L^2}+\f{1}{\nu}\|\Lambda\D_q\tilde{\vr}\|_{L^2}^2-\f{3}{\nu^2}(\Lambda\D_q\tilde{\vr} \, |\D_q\tilde{h})=\mathcal{Y}_1-KV'(t)f_q^2
\end{equation}
with
\begin{equation*}
\begin{split}
\mathcal{Y}_1&=-\f{3}{\nu^2}(\D_q(\u\cdot\nabla\tilde{h})|\D_q\tilde{h})-(\Lambda\D_q(\u\cdot\nabla\tilde{\vr})|\Lambda\D_q{\tilde{\vr}})+\f{3}{\nu^2}(\D_q\tilde{\mathcal{K}}|\D_q\tilde{h})\\
&\quad
+(\Lambda\D_q\tilde{\mathcal{J}}|\Lambda\D_q\tilde{\vr})+\f{1}{\nu}(\D_q(\u\cdot\nabla\tilde{h})|\Lambda\D_q\tilde{\vr})+\f{1}{\nu}(\Lambda\D_q(\u\cdot\nabla\tilde{\vr})|\D_q\tilde{h})\\
&\quad
-\f{1}{\nu}(\Lambda\D_q\tilde{\mathcal{J}}|\D_q\tilde{h})-\f{1}{\nu}(\D_q\tilde{\mathcal{K}}|\Lambda\D_q\tilde{\vr}).
\end{split}
\end{equation*}
 For $\mathcal{Y}_1$,  the assumption $q\geq q_0$ implies that $2^{-q(\phi^{\frac12, \frac32}(q)-\frac32)}\leq 1$, then using Lemma \ref{cl}, we can see that
\begin{equation}\label{1310}
\begin{split}
|\mathcal{Y}_1| &\leq \frac{3C}{\nu^2}\alpha_q2^{-\frac{q}{2}}V'\|\tilde{h}\|_{B^\frac12}\|\D_q \tilde{h}\|_{L^2}+ C\alpha_q2^{-q(\phi^{\frac12, \frac32}(q)-1)}V'\|\tilde{\vr}\|_{\tilde{B}^{\frac12, \frac32}}\|\Lambda \D_q\tilde{\vr}\|_{L^2}
\\
&\ \ +\frac{3}{\nu^2}\|\D_q\tilde{\mathcal{K}}\|_{L^2}\|\D_q \tilde{h}\|_{L^2}+ \|\Lambda\D_q\tilde{\mathcal{J}}\|_{L^2}\|\Lambda\D_q\tilde{\vr}\|_{L^2}\\
&\ \ +\frac{1}{\nu}( \|\Lambda\D_q\tilde{\mathcal{J}}\|_{L^2}\|\D_q\tilde{h}\|_{L^2}+\|\D_q\tilde{\mathcal{K}}\|_{L^2}\|\Lambda\D_q\tilde{\vr}\|_{L^2})
\\
&\ \ +\frac{C\alpha_q}{\nu}V'(2^{-\frac{q}{2}}\|\Lambda \D_q \tilde{\vr}\|_{L^2}\|\tilde{h}\|_{B^\frac12}+ 2^{-q(\phi^{\frac12, \frac32}(q)-1)}\|\tilde{\vr}\|_{\tilde{B}^{\frac12, \frac32}}\|\D_q\tilde{h}\|_{L^2})
\\
&\le
Cf_q\Big(\|\D_q\tilde{\mathcal{K}}\|_{L^2}+\|\Lambda\D_q\tilde{\mathcal{J}}\|_{L^2}+2^{-\f{q}{2}}\alpha_q
V'(\|\tilde{\vr}\|_{\tilde{B}^{\f{1}{2},\f{3}{2}}}+\|\tilde{h}\|_{B^{\f{1}{2}}})\Big).
\end{split}
\end{equation}
Since $q\geq q_0=\log_2\left(\f{3}{\nu}\right)$, the Bernstein's inequality implies
that
$$ \frac{3}{\nu^2}|(\Lambda\D_q\tilde{\vr}|\D_q\tilde{h})|\leq 2^{-q_0}\frac{3}{\nu^2}\|\Lambda\D_q\tilde{\vr}\|_{L^2}\|\Lambda\D_q\tilde{h}\|_{L^2}\leq \frac1{2\nu}(\|\Lambda\D_q\tilde{\vr}\|_{L^2}^2+\|\Lambda\D_q\tilde{h}\|_{L^2}^2).
$$
As a result,
\begin{equation}\label{1391}
\begin{split}
\|\Lambda\D_q\tilde{\vr}\|_{L^2}^2+\|\D_q\tilde{h}\|_{L^2}^2&\leq \|\Lambda\D_q\tilde{\vr}\|_{L^2}^2+2^{-2q_0}\|\Lambda\D_q\tilde{h}\|_{L^2}^2\\
& \leq \|\Lambda\D_q\tilde{\vr}\|_{L^2}^2+\frac{\nu^2}{9}\|\Lambda\D_q\tilde{h}\|_{L^2}^2\\
& \leq C\left(\f{2}{\nu}\|\Lambda\D_q\tilde{h}\|^2_{L^2}+\f{1}{\nu}\|\Lambda\D_q\tilde{\vr}\|_{L^2}^2-\frac{3}{\nu^2}(\Lambda\D_q\tilde{\vr}|\D_q\tilde{h})\right).\\
\end{split}
\end{equation}
 Hence, combining \eqref{139}, \eqref{1310} and \eqref{1391}
together, there exists a positive constant $\kappa$ such that
\begin{equation}\label{1311}
\begin{split}
&\f{1}{2}\f{d}{dt}f_q^2+\kappa\Big(\|\D_q\tilde{h}\|^2_{L^2}+\|\Lambda\D_q\tilde{\vr}\|_{L^2}^2\Big)\\
&\ \ \le
Cf_q\Big(\|\D_q\tilde{\mathcal{K}}\|_{L^2}+\|\Lambda\D_q\tilde{\mathcal{J}}\|_{L^2}+2^{-\f{q}{2}}\alpha_q
V'(\|\tilde{\vr}\|_{\tilde{B}^{\f{1}{2},\f{3}{2}}}+\|\tilde{h}\|_{B^{\f{1}{2}}})\Big)-KV'(t)f_q^2.
\end{split}
\end{equation}

{\bf Step 3: Damping Effect.}\quad We now show
that inequalities \eqref{137} and \eqref{1311} entail a decay for
$h$ and $\vr$. Denote $g_q=2^{\f{q}{2}}f_q$ for $q\in \mathbb{Z}$. It follows from \eqref{137}, \eqref{1311}, and
Bernstein's inequality that
\begin{equation}\label{1312}
\begin{split}
\f{1}{2}\f{d}{dt}(g_q^2)+\kappa 2^{q\phi^{2,0}(q-q_0)}g_q^2&\le
C\alpha_q
g_q\Big(\|\tilde{\mathcal{K}}\|_{B^{\f{1}{2}}}+\|\tilde{\mathcal{J}}\|_{\tilde{B}^{\f{1}{2},\f{3}{2}}}+
V'(\|\tilde{\vr}\|_{\tilde{B}^{\f{1}{2},\f{3}{2}}}+\|\tilde{h}\|_{B^{\f{1}{2}}})\Big)\\
&\quad-KV'(t)g_q^2.
\end{split}
\end{equation}
Let $\dl>0$ be a small parameter (which will tend to 0) and denote
$\chi_q^2=g_q^2+\dl^2.$ From \eqref{1312}, dividing by
$\chi_q$, we obtain
\begin{equation*}
\begin{split}
\f{d}{dt}\chi_q+\kappa 2^{q\phi^{2,0}(q-q_0)} \chi_q &\le
C\alpha_q\Big(\|\tilde{\mathcal{K}}\|_{B^{\f{1}{2}}}+\|\tilde{\mathcal{J}}\|_{\tilde{B}^{\f{1}{2},\f{3}{2}}}+
V'(\|\tilde{\vr}\|_{\tilde{B}^{\f{1}{2},\f{3}{2}}}+\|\tilde{h}\|_{B^{\f{1}{2}}})\Big)\\&\quad-KV'\chi_q+\dl
KV'+\dl \kappa 2^q.
\end{split}
\end{equation*}
Integrating the above inequality over $[0,t]$ and having $\dl$
tend to $0$, we obtain,
\begin{equation}\label{1313}
\begin{split}
g_q(t)+\kappa 2^{q\phi^{2,0}(q-q_0)}\int_0^t g_q(s)ds&\le
g_q(0)+C\int_0^t\alpha_q(s)\Big(\|\tilde{\mathcal{K}}\|_{B^{\f{1}{2}}}+\|\tilde{\mathcal{J}}\|_{\tilde{B}^{\f{1}{2},\f{3}{2}}}\Big)ds\\
&\quad+\int_0^tV'\Big(C\alpha_q(s)
(\|\tilde{\vr}\|_{\tilde{B}^{\f{1}{2},\f{3}{2}}}+\|\tilde{h}\|_{B^{\f{1}{2}}})-Kg_q(s)\Big)ds.
\end{split}
\end{equation}
Bernstein's inequality implies
$$C_0g_q\le
2^{\f{q}{2}}\|\D_q\tilde{h}\|_{L^2}+2^{q\phi^{\f{1}{2},\f{3}{2}}(q-q_0)}\|\D_q\tilde{\r}\|_{L^2}\le\f{1}{C_0}g_q$$
for some universal positive constant $C_0$, and hence
\begin{equation*}
\begin{split}
C\alpha_q(s)
(\|\tilde{\vr}\|_{\tilde{B}^{\f{1}{2},\f{3}{2}}}+\|\tilde{h}\|_{B^{\f{1}{2}}})-Kg_q(s)&\le
C\alpha_q(s)\|\tilde{h}\|_{B^{\f{1}{2}}}-C_0K
2^{\f{q}{2}}\|\D_q\tilde{h}\|_{L^2}\\
&\quad+C\alpha_q(s)\|\tilde{\vr}\|_{\tilde{B}^{\f{1}{2},\f{3}{2}}}-C_0K2^{q\phi^{\f{1}{2},\f{3}{2}}(q-q_0)}\|\D_q\tilde{\vr}\|_{L^2}.
\end{split}
\end{equation*}
If we choose $KC_0>C$, we have
\begin{equation}\label{1314}
\sum_{q\in\mathbb{Z}}\Big\{C\alpha_q(s)
\left(\|\tilde{\vr}\|_{\tilde{B}^{\f{1}{2},\f{3}{2}}}+\|\tilde{h}\|_{B^{\f{1}{2}}}\right)-Kg_q(s)\Big\}\le
0.
\end{equation}
With the inequality \eqref{1314} in hand, after summation over
$\mathbb{Z}$, we deduce from \eqref{1313} that
\begin{equation}\label{1315}
\begin{split}
&\|\tilde{h}(t)\|_{B^{\f{1}{2}}}+\|\tilde{\vr}(t)\|_{\tilde{B}^{\f{1}{2},\f{3}{2}}}+\kappa\int_0^t\left(\|\tilde{h}(\tau)\|_{\tilde{B}^{\f52,\f{1}{2}}}+\|\tilde{\vr}(\tau)\|_{\tilde{B}^{\f{5}{2},\f{3}{2}}}
\right)d\tau
\\&\quad\le C\Big\{\|h_0\|_{B^{\f{1}{2}}}+\|\vr_0\|_{\tilde{B}^{\f{1}{2},\f{3}{2}}}
+\int_0^t\Big(\|\tilde{\mathcal{K}}(s)\|_{\tilde{B}^{\f{1}{2}}}+\|\tilde{\mathcal{J}}(s)\|_{\tilde{B}^{\f{1}{2},\f{3}{2}}}\Big)ds\Big\}.
\end{split}
\end{equation}

{\bf Step 4: Smoothing Effect of $h$.}\quad Based on the damping effect for $\vr$, we can now further get the smoothing effect of $h$ by considering \eqref{2e} with $\Lambda\vr$ being seen as a source term.
Indeed, thanks to \eqref{1315}, it suffices to state the proof for
high frequencies only. We therefore assume that $q\ge q_0$.

Define
$I_q=2^{\f{q}{2}}\|\D_q\tilde{h}\|_{L^2}$.
Then, from the energy estimates for the system
\begin{equation*}
\partial_t \D_q\tilde{h}-\nu\D \D_q\tilde{h}=-\D_q(\u\cdot\nabla \tilde{h})+\Lambda\D_q \tilde{\vr}
+\D_q\tilde{\mathcal{K}}-KV'(t)\D_q\tilde{h},
\end{equation*}
we have
\begin{equation*}
\begin{split}
\f{1}{2}\f{d}{dt}I_q^2+\nu 2^{2q}I_q^2&\le
I_q\Big(2^{\f{3q}{2}}\|\D_q\tilde{\vr}\|_{L^2}+2^{\f{q}{2}}\|\D_q\tilde{\mathcal{K}}\|_{L^2}\Big)
+CI_qV'(t)\alpha_q\|\tilde{h}\|_{B^{\f{1}{2}}},
\end{split}
\end{equation*}
for a universal positive constant $\kappa$. Using
$J_q^2=I_q^2+\dl^2$, integrating over $[0,t]$ and then taking the
limit as $\dl\rightarrow 0$, we deduce
\begin{equation}\label{1322}
\begin{split}
I_q(t)+\nu 2^{2q}\int_0^tI_q(s)ds &\le I_q(0)+\int_0^t2^{\f{q}{2}}\|\D_q\tilde{\mathcal{K}}(s)\|_{L^2}ds+\int_0^t2^{\f{3q}{2}}\|\D_q\tilde{\vr}(s)\|_{L^2}ds\\&\qquad
+C\int_0^tV'(s)\alpha_q(s)\|\tilde{h}(s)\|_{B^{\f{1}{2}}}ds.
\end{split}
\end{equation}
We therefore get
\begin{equation*}
\begin{split}
&\sum_{q\ge
q_0}2^{\f{q}{2}}\|\D_q\tilde{h}(t)\|_{L^2}+\nu\int_0^t\sum_{q\ge
q_0}2^{\f{5q}{2}}\|\D_q\tilde{h}(s)\|_{L^2}ds\\
&\quad\le\|h_0\|_{B^{\f{1}{2}}}+\int_0^t\|\tilde{\mathcal{K}}(s)\|_{B^{\f{1}{2}}}ds+\int_0^t\sum_{q\ge
q_0}2^{\f{3q}{2}}\|\D_q\tilde{\vr}(s)\|_{L^2}ds+CV(t)\sup_{s\in
[0,t]}\|\tilde{h}\|_{B^{\f{1}{2}}}.
\end{split}
\end{equation*}
Using \eqref{1315}, we eventually conclude that
\begin{equation*}
\begin{split}
\nu\int_0^t\sum_{q\ge
q_0}2^{\f{5q}{2}}\|\D_q\tilde{h}(s)\|_{L^2}ds&\le
C(1+V(t))\Big(\|\vr_0\|_{\tilde{B}^{\f{1}{2},\f{3}{2}}}+\|h_0\|_{B^{\f{1}{2}}}\\&\qquad
+\int_0^t\Big(\|\tilde{\mathcal{K}}(s)\|_{B^{\f{1}{2}}}
+\|\tilde{\mathcal{J}}(s)\|_{\tilde{B}^{\f{1}{2},\f{3}{2}}}\Big)ds\Big).
\end{split}
\end{equation*}
Combining the last inequality with \eqref{1315}, we finish the
proof of Proposition \ref{p2}.
\end{proof}

\bigskip

\section{Global Existence for Initial Data Near Equilibrium}

In this section, we are going to show that if the initial data
$$\|\r_0-1\|_{\tilde{B}^{\f{1}{2},\f{3}{2}}}+\|\u_0\|_{B^{\f{1}{2}}}+\|\d_0-\hat{\d}\|_{\tilde{B}^{\f{1}{2},\f{3}{2}}}\le
\eta$$ for some sufficiently small $\eta$, there exists a positive
constant $\Gamma$ such that
$$\|(\r-1, \u,\d-\hat{\d})\|_{\mathfrak{B}^{\f{3}{2}}}\le\Gamma\eta.$$ This uniform estimate will enable us to extend the local solution $(\r, \u, \d)$ obtained within an iterative scheme as in \cite{RD2} to be a global one. To this
end, we use a contradiction argument. Define
$$T_0=\sup\left\{T\in[0,\infty): \|(\r-1,\u,\d-\hat{\d})\|_{\mathfrak{B}_T^{\f{3}{2}}}\le
\Gamma\eta\right\},$$ with $\Gamma$ to be determined later.
Suppose that $T_0<\infty$.
We apply the linear estimates in Proposition \ref{p1} and Proposition \ref{p2} to the solution of reformulated system $\eqref{e2}$ such that for all $t\in [0,T_0]$, the following estimate holds:
\begin{equation}\label{316}
\begin{split}
& \|\d(t)-\hat{\d}\|_{L^\infty_{T_0}(\tilde{B}^{\f{1}{2},\f{3}{2}})}
+\|\O(t)\|_{L^\infty_{T_0}(B^{\f{1}{2}})}+\int_0^{T_0}\left(\|\d(s)-\hat{\d}\|_{\tilde{B}^{\f{5}{2},\f{7}{2}}}+\|\O(s)\|_{B^{\f{5}{2}}}\right)ds \\
&\ \ \le
Ce^{CV}\left(\|\d_0-\hat{\d}\|_{\tilde{B}^{\f{1}{2},\f{3}{2}}}+\|\O_0\|_{B^{\f{1}{2}}}
+\|\mathcal{L}\|_{L^1_{T_0}(B^{\f{1}{2}})}+\|\mathcal{M}\|_{L^1_{T_0}(\tilde{B}^{\f{1}{2},\f{3}{2}})}\right),
\end{split}
\end{equation}
and
\begin{equation}\label{1316}
\begin{split}
&\|\r(t)-1\|_{L^\infty_{T_0}(\tilde{B}^{\f{1}{2},\f{3}{2}})}
+\|h(t)\|_{L^\infty_{T_0}(B^{\f{1}{2}})}+\int_0^{T_0}\left(\|\r(s)-1\|_{\tilde{B}^{\f{5}{2},\f{3}{2}}}+\|h(s)\|_{B^{\f{5}{2}}}\right)ds\\&\quad\le
Ce^{CV}\left(\|\r_0-1\|_{\tilde{B}^{\f{1}{2},\f{3}{2}}}+\|h_0\|_{B^{\f{1}{2}}}
+\|\mathcal{K}\|_{L^1_{T_0}(B^{\f{1}{2}})}+\|\mathcal{J}\|_{L^1_{T_0}(\tilde{B}^{\f{1}{2},\f{3}{2}})}\right),
\end{split}
\end{equation}
where
\begin{equation*}
V=\int_0^{T_0} \|\u\|_{B^{\f{5}{2}}}ds.
\end{equation*}
We note that $\nabla\d=\nabla(\d-\hat{\d})$ because $\hat\d\in S^2$ is a constant vector. As a result, the function $\mathcal{N}$ (see \eqref{N}) can be rewritten as
$$\mathcal{N}=\widehat{\mathcal{N}}:=-\u\cdot\nabla\u-\f{\r-1}{\r}\mathcal{A}\u-\f{1}{\r}\Dv\left(\nabla(\d-\hat{\d})\odot\nabla(\d-\hat{\d})-\f12|\nabla(\d-\hat{\d})|^2I\right).$$
Therefore, the functions $\mathcal{L}$ and $\mathcal{M}$ in \eqref{316} are given by
\begin{equation*}
\begin{split}
\mathcal{L}&=\Lambda^{-1}\textrm{curl}\widehat{\mathcal{N}}+\u\cdot\nabla\O,
\\
\mathcal{M}&=|\nabla(\d-\hat{\d})|^2(\d-\hat{\d})+|\nabla(\d-\hat{\d})|^2\hat{\d},
\end{split}
\end{equation*}
while for the functions $\mathcal{K}$ and $\mathcal{J}$ in \eqref{1316}, we have
\begin{equation*}
\begin{split}
\mathcal{K}&=\Lambda^{-1}\Dv\widehat{\mathcal{N}}+\u\cdot\nabla h,
\\
\mathcal{J}&=-(\r-1)\Dv\u.
\end{split}
\end{equation*}

In what follows, we derive estimates for the nonlinear terms  $\mathcal{L}$, $\mathcal{M}$, $\mathcal{K}$ and $\mathcal{J}$. Indeed,
for the term $\u\cdot\nabla\O$, we infer from Proposition \ref{p22} that
\begin{equation}\label{317}
\begin{split}
\|\u\cdot\nabla\O\|_{L^1_{T_0}(B^{\f{1}{2}})}&\le
C\int_0^{T_0}\|\u\|_{B^{\f{1}{2}}}\|\nabla\O\|_{B^{\f{3}{2}}}dt\\
&\le C\|\u\|_{L^\infty_{T_0}(B^{\f{1}{2}})}\|\O\|_{L^1_{T_0}(B^{\f{5}{2}})}\le
C\Gamma^2\eta^2.
\end{split}
\end{equation}
Similarly, we have
\begin{equation}\label{318}
\begin{split}
\|\Lambda^{-1}\textrm{curl}\u\cdot\nabla\u\|_{L^1_{T_0}(B^{\f{1}{2}})}&\le
C\|\u\cdot\nabla\u\|_{L^1_{T_0}(B^{\f{1}{2}})}\le
\int_0^{T_0}\|\u\|_{B^{\f{1}{2}}}\|\nabla\u\|_{B^{\f{3}{2}}}dt\\
&\le C\|\u\|_{L^\infty_{T_0}(B^{\f{1}{2}})}\|\u\|_{L^1_{T_0}(B^{\f{5}{2}})}\le
C\Gamma^2\eta^2.
\end{split}
\end{equation}
By the embedding $B^\frac32\hookrightarrow L^\infty$, we consider $F(\vr)=\frac{\vr}{1+\vr}$ with $\vr=\r-1$, then by \cite[Lemma 1.6]{RD2}, we have
$$\|F(\vr)\|_{B^\frac32}\leq C_0(\|\vr\|_{L^\infty})\|\vr\|_{B^\frac32},$$
where $\|\vr\|_{L^\infty}\leq 1+\|\r\|_{L^\infty}\leq 1+C\|\rho\|_{B^\frac32}$. Then we can apply Proposition \ref{p22} to obtain that
\begin{equation}\label{319}
\begin{split}
&\left\|\Lambda^{-1}\textrm{curl}\left(\f{1}{\r}\Dv\left(\nabla(\d-\hat{\d})\odot\nabla(\d-\hat{\d})-\f12|\nabla(\d-\hat{\d})|^2I\right)\right)\right\|_{L^1_{T_0}(B^{\f{1}{2}})}\\
&\quad\le C\left(1+\left\|\f{\r-1}{\r}\right\|_{L^\infty_{T_0}(B^{\f32})}\right)\left\|\Dv\left(\nabla(\d-\hat{\d})\odot\nabla(\d-\hat{\d})-\f12|\nabla(\d-\hat{\d})|^2I\right)\right\|_{L^1_{T_0}(B^{\f{1}{2}})}\\
&\quad\le C\Big(1+\|\r-1\|_{L^\infty_{T_0}(B^{\f32})}\Big)\|\nabla(\d-\hat{\d})\odot\nabla(\d-\hat{\d})-\f12|\nabla(\d-\hat{\d})|^2I\|_{L^1_{T_0}(B^{\f{3}{2}})}\\
&\quad\le C\Big(1+\|\r-1\|_{L^\infty_{T_0}(\tilde{B}^{\f12,\f32})}\Big)\int_0^{T_0}\|\nabla(\d-\hat{\d})\|^2_{B^{\f{3}{2}}}dt\\
&\quad\le
C(1+\Gamma\eta)\int_0^{T_0}\|\nabla(\d-\hat{\d})\|^2_{\tilde{B}^{\f{1}{2},\f{3}{2}}}dt\\
&\quad\le
C(1+\Gamma\eta)\|\d-\hat{\d}\|^2_{L^2_{T_0}(\tilde{B}^{\f{3}{2},\f{5}{2}})}\\
&\quad\le
C(1+\Gamma\eta)\|\d-\hat{\d}\|_{L^\infty_{T_0}(\tilde{B}^{\f{1}{2},\f{3}{2}})}\|\d-\hat{\d}\|_{L^1_{T_0}(\tilde{B}^{\f{5}{2},\f{7}{2}})}\le
C(1+\Gamma\eta)\Gamma^2\eta^2.
\end{split}
\end{equation}
In the last line of above inequality we have used the interpolation
$$\|f\|^2_{L^2_{T_0}(\tilde{B}^{\f{3}{2},\f{5}{2}})}\le
C\|f\|_{L^\infty_{T_0}(\tilde{B}^{\f{1}{2},\f{3}{2}})}\|f\|_{L^1_{T_0}(\tilde{B}^{\f{5}{2},\f{7}{2}})}.$$
The term $\f{\r-1}{\r}\mathcal{A}\u$ can be dealt with in a similar way such that
\begin{equation}\label{1319}
 \begin{split}
\left\|\Lambda^{-1}\textrm{curl}\left(\f{\r-1}{\r}\mathcal{A}\u\right)\right\|_{L^1_{T_0}(B^{\f{1}{2}})}&\le\left\|\f{\r-1}{\r}\mathcal{A}\u\right\|_{L^1_{T_0}(B^{\f{1}{2}})}\\
&\le \left\|\f{\r-1}{\r}\right\|_{L^\infty_{T_0}(B^{\f32})}\|\mathcal{A}\u\|_{L^1_{T_0}(B^{\f{1}{2}})}\\
&\le C\|\r-1\|_{L^\infty_{T_0}(B^{\f32})}\|\mathcal{A}\u\|_{L^1_{T_0}(B^{\f{1}{2}})}\\
&\le C\|\r-1\|_{L^\infty_{T_0}(\tilde{B}^{\f12,\f32})}\|\u\|_{L^1_{T_0}(B^{\f{5}{2}})}\le C\Gamma^2\eta^2.
 \end{split}
\end{equation}
Combining \eqref{317}--\eqref{1319} together, we
obtain
\begin{equation}\label{320}
\|\mathcal{L}\|_{L^1_{T_0}(B^{\f{1}{2}})}\le C\Gamma^2\eta^2+C\Gamma^3\eta^3.
\end{equation}
Similarly, we can derive the estimate for $\mathcal{K}$:
\begin{equation}\label{1320}
\|\mathcal{K}\|_{L^1_{T_0}(B^{\f{1}{2}})}\le C\Gamma^2\eta^2+C\Gamma^3\eta^3.
\end{equation}

Next, we turn to the estimate for $\mathcal{M}$. For
$|\nabla(\d-\hat{\d})|^2(\d-\hat{\d})$, indeed, we have, as in
\eqref{319}
\begin{equation}\label{321}
\begin{split}
&\||\nabla(\d-\hat{\d})|^2(\d-\hat{\d})\|_{L^1_{T_0}(\tilde{B}^{\f{1}{2},\f{3}{2}})}\\
&\quad\le
C\||\nabla(\d-\hat{\d})|^2\|_{L^1_{T_0}(B^{\f{3}{2}})}\|(\d-\hat{\d})\|_{L^\infty_{T_0}(\tilde{B}^{\f{1}{2},\f{3}{2}})}\\
&\quad\le C\Gamma\eta\int_0^{T_0}\|\nabla(\d-\hat{\d})\|^2_{B^{\f{3}{2}}}dt\le
C\Gamma\eta\int_0^{T_0}\|\nabla(\d-\hat{\d})\|^2_{\tilde{B}^{\f{1}{2},\f{3}{2}}}dt\\
&\quad\le
C\Gamma\eta\|(\d-\hat{\d})\|^2_{L^2_{T_0}(\tilde{B}^{\f{3}{2},\f{5}{2}})}\\
&\quad\le
C\Gamma\eta\|(\d-\hat{\d})\|_{L^\infty_{T_0}(\tilde{B}^{\f{1}{2},\f{3}{2}})}\|(\d-\hat{\d})\|_{L^1_{T_0}(\tilde{B}^{\f{5}{2},\f{7}{2}})}\le
C\Gamma^3\eta^3.
\end{split}
\end{equation}
For $|\nabla(\d-\hat{\d})|^2\hat{\d}$, we have, by the
definition of Besov's spaces
\begin{equation}\label{3211}
\begin{split}
\||\nabla(\d-\hat{\d})|^2\hat{\d}\|_{L^1_{T_0}(\tilde{B}^{\f{1}{2},\f{3}{2}})}&\le
C\||\nabla(\d-\hat{\d})|^2\|_{L^1_{T_0}(\tilde{B}^{\f{1}{2},\f{3}{2}})}\\
&\le C\int_0^{T_0}\|\nabla(\d-\hat{\d})\|_{L^\infty}\|\nabla(\d-\hat{\d})\|_{\tilde{B}^{\f{1}{2},\f{3}{2}}}dt\\
&\le
C\int_0^{T_0}\|\nabla(\d-\hat{\d})\|^2_{\tilde{B}^{\f{1}{2},\f{3}{2}}}dt\\
&\le C\|\d-\hat{\d}\|^2_{L^2_{T_0}(\tilde{B}^{\f{3}{2},\f{5}{2}})}\\
&\le
C\|\d-\hat{\d}\|_{L^\infty_{T_0}(\tilde{B}^{\f{1}{2},\f{3}{2}})}\|\d-\hat{\d}\|_{L^1_{T_0}(\tilde{B}^{\f{5}{2},\f{7}{2}})}\le
C\Gamma^2\eta^2.
\end{split}
\end{equation}
Finally, for $\mathcal{J}$, we have
\begin{equation}\label{32111}
\begin{split}
\|(\r-1)\Dv\u\|_{L^1_{T_0}(\tilde{B}^{\f{1}{2},\f{3}{2}})}&\le
C\|\r-1\|_{L^\infty_{T_0}(\tilde{B}^{\f12,\f32})}\|\Dv\u\|_{L^1_{T_0}(B^{\f32})}\\
&\le C\Gamma^2\eta^2.
\end{split}
\end{equation}
Substituting \eqref{320}--\eqref{32111} back to \eqref{316} and
\eqref{1316}, we obtain
\begin{equation}\label{322}
\begin{split}
\|(\r-1,\u,\d-\hat{\d}))\|_{\mathfrak{B}_{T_0}^{\f{3}{2}}}\leq C_1e^{C_1\Gamma\eta}\left(\eta+\Gamma^2\eta^2+\Gamma^3\eta^3\right).
\end{split}
\end{equation}
We choose $\Gamma=4C_1$, and then $\eta>0$ satisfying
\begin{equation}
e^{C_1\Gamma \eta}< 2, \quad \Gamma^2\eta\leq \frac12,\quad
\Gamma^3\eta^2\leq \frac12.
\end{equation}
Hence, it follows from \eqref{322} and the above choices of $\Gamma$ and $\eta$ that
$$\|(\r-1,\u,\d-\hat{\d})\|_{\mathfrak{B}_{T_0}^{\f{3}{2}}}< \Gamma\eta,$$
which is a contradiction with the definition of $T_0$. As a consequence, we can conclude that
$T_0=\infty$. The proof of global existence is thus proved.

\bigskip

\section{Uniqueness}

In this section, we will address the uniqueness of the solution in
$\mathfrak{B}^{\f{3}{2}}$. For this purpose, suppose that $(\r_i,\u_i,
\d_i)_{i=1,2}$ in $\mathfrak{B}^{\f{3}{2}}$ solve \eqref{e1} with
the same initial data.

Define $$\dl\r=\r_2-\r_1, \ \ \dl\u=\u_2-\u_1,\ \  \dl h=\Lambda^{-1}\Dv\dl\u,\ \
\dl\O=\Lambda^{-1}\textrm{curl}\dl\u,\ \ \dl \d=\d_2-\d_1.$$
then $(\dl\r, \dl\u, \dl \d)\in \mathfrak{B}_T^{\f{3}{2}}$ for all $T>0$.
On the other hand, since $(\r_i, \u_i, \d_i)_{i=1,2}$ are solutions to
\eqref{e1} with the same initial data, $(\dl\r, \dl\u, \dl \d)$ solves
\begin{equation}\label{52}
\begin{cases}
\partial_t\dl\r+\u_2\cdot\nabla\dl\r+\Lambda\dl h=\dl \mathcal{J},\\
\partial_t\dl h+\u_2\cdot\nabla\dl h-\nu\D\dl h-\Lambda\dl\r=\dl\mathcal{K},\\
\partial_t\dl\O+\u_2\cdot\nabla\dl\O-\D\dl\O=\dl \mathcal{L},\\
\partial_t\dl \d+\u_2\cdot\nabla\dl\d-\D\dl\d=\dl\mathcal{M},\\
\dl\u=-\Lambda^{-1}\nabla\dl h+\Lambda^{-1}\textrm{curl}\dl\O,
\end{cases}
\end{equation}
with
$$\dl \mathcal{J}=-\dl\u\cdot\nabla\r_1-\dl\r\Dv\u_2-(\r_1-1)\Dv\dl\u,$$
\begin{equation*}
\begin{split}
\dl\mathcal{K}&=\u_2\cdot\nabla\dl h+\Lambda^{-1}\Dv\Big(-\u_2\cdot\nabla\dl
\u-\dl\u\cdot\nabla\u_1-\left(\f{1}{\r_1}-\f{1}{\r_2}\right)\mathcal{A}\u_2+\left(\f{1}{\r_1}-1\right)\mathcal{A}\dl\u\\
&\quad-\f{1}{\r_1}\Dv\Big(\nabla(\d_2-\hat{\d})\odot\nabla\dl\d+\nabla\dl\d\odot\nabla(\d_1-\hat{\d})\\&\qquad-\f12I\Big(\nabla (\d_2-\hat{\d}):\nabla\dl\d+\nabla\dl\d:\nabla(\d_1-\hat{\d})\Big)\Big)\\
&\quad-\left(\f{1}{\r_1}-\f{1}{\r_2}\right)\Dv\Big(\nabla(\d_2-\hat{\d})\odot\nabla(\d_2-\hat{\d})-\f12|\nabla(\d_2-\hat{\d})|^2I\Big)\Big),
\end{split}
\end{equation*}
\begin{equation*}
\begin{split}
\dl\mathcal{L}&=\u_2\cdot\nabla\dl \O+\Lambda^{-1}\textrm{curl}\Big(-\u_2\cdot\nabla\dl
\u-\dl\u\cdot\nabla\u_1-\left(\f{1}{\r_1}-\f{1}{\r_2}\right)\mathcal{A}\u_2+\f{1}{\r_1}\mathcal{A}\dl\u\\
&\quad-\f{1}{\r_1}\Dv\Big(\nabla(\d_2-\hat{\d})\odot\nabla\dl\d+\nabla\dl\d\odot\nabla(\d_1-\hat{\d})\\&\qquad-\f12I\Big(\nabla (\d_2-\hat{\d}):\nabla\dl\d+\nabla\dl\d:\nabla(\d_1-\hat{\d})\Big)\Big)\\
&\quad-\left(\f{1}{\r_1}-\f{1}{\r_2}\right)\Dv\Big(\nabla(\d_2-\hat{\d})\odot\nabla(\d_2-\hat{\d})-\f12|\nabla(\d_2-\hat{\d})|^2I\Big)\Big),
\end{split}
\end{equation*}
and
\begin{equation*}
\begin{split}
\dl\mathcal{M}&=|\nabla(\d_1-\hat{\d})|^2\dl\d+(\nabla\dl\d:\nabla(\d_1-\hat{\d})+\nabla(\d_2-\hat{\d}):\nabla\dl\d)\d_2,
\end{split}
\end{equation*}
where we used the notation $A:B=\sum_{i,j=1}^3A_{ij}B_{ij}$.

Applying Proposition \ref{p1} and Proposition \ref{p2} to the system \eqref{52}, we get
\begin{equation}\label{1531}
 \begin{split}
&\|(\dl\r,\dl\u,\dl\d)\|_{\mathfrak{B}_T^{\f12}}\\
&\lesssim \exp\Big(C\|\u_2\|_{L^1_T(B^{\f52})}\Big)\Big(\|\dl\mathcal{J}\|_{L_T^1(\tilde{B}^{-\f{1}{2},\f{1}{2}})}
+\|\dl\mathcal{K}\|_{L_T^1(B^{-\f{1}{2}})} +\|\dl\mathcal{L}\|_{L_T^1(B^{-\f{1}{2}})}\\
&\quad+\|\dl\mathcal{M}\|_{L_T^1(\tilde{B}^{-\f{1}{2},\f{1}{2}})}\Big).
 \end{split}
\end{equation}

Since $(\dl\r, \dl\u, \dl\d)\in \mathfrak{B}_T^{\f{3}{2}}$, $\partial_t(\r_i-1)\in L^1_{loc}(B^{\f12})$, and hence $\r_i-1\in C(B^{\f12})\cap L^\infty(B^{\f32})$. This entails $\r_i-1\in C([0,\infty)\times\R^3)$. On the other
hand, if $\eta$ is sufficiently small, we have
$$|\r_1(t,x)-1|\le \f14
\quad\textrm{for all}\quad t\ge 0\textrm{ and }x\in\R^3.$$
Continuity in time for $\r_2-1$ thus yields the existence of a time $T>0$ such that
$$\|\r_i(t)-1\|_{L^\infty}\le \f12 \quad\textrm{for}\quad i=1,2\textrm{  and  }t\in[0,T].$$
Repeating the argument in Section 5, we easily infer that
\begin{equation*}
 \begin{split}
\|\dl\mathcal{J}\|_{L_T^1(\tilde{B}^{-\f{1}{2},\f{1}{2}})}\lesssim \|\r_1-1\|_{L^\infty_T(\tilde{B}^{\f12,\f32})}\|\dl\u\|_{L^1_T(B^{\f32})}+\|\Dv\u_2\|_{L^1_T(B^{\f32})}\|\dl\r\|_{L^\infty_T(\tilde{B}^{-\f{1}{2},\f{1}{2}})},
 \end{split}
\end{equation*}
\begin{equation*}
 \begin{split}
&\|\dl\mathcal{K}\|_{L_T^1(B^{-\f{1}{2}})}+\|\dl\mathcal{L}\|_{L_T^1(B^{-\f{1}{2}})}\\&\quad\lesssim \|\u_2\|_{L^2_T(B^{\f32})}\|\nabla\dl\u\|_{L^2_T(B^{-\f12})}+\|\dl\u\|_{L^2_T(B^{\f12})}\|\nabla\u_1\|_{L^2_T(B^{\f12})}
\\
&\qquad+\Big(1+\|\r_1-1\|_{L^\infty_T(B^{\f32})}+\|\r_2-1\|_{L^\infty_T(B^{\f32})}\Big)\|\nabla^2\u_2\|_{L^1_T(B^{\f12})}\|\dl\r\|_{L^\infty_T(B^{\f12})}\\
&\qquad+\Big(1+\|\r_1-1\|_{L^\infty_T(B^{\f32})}\Big)\|\nabla^2\dl\u\|_{L^1_T(B^{-\f12})}\\
&\qquad+\Big(1+\|\r_1-1\|_{L^\infty_T(B^{\f32})}+\|\r_2-1\|_{L^\infty_T(B^{\f32})}\Big)\||\nabla(\d_2-\hat{\d})|^2\|_{L^1_T(B^{\f32})}\|\dl\r\|_{L^\infty_T(B^{\f12})}\\
&\qquad+\Big(1+\|\r_1-1\|_{L^\infty_T(B^{\f32})}\Big)\|\dl\d\|_{L^2_T(\tilde{B}^{\f12,\f32})}\Big(\|\d_1-\hat{\d}\|_{L^2_T(\tilde{B}^{\f32,\f52})}+\|\d_2-\hat{\d}\|_{L^2_T(\tilde{B}^{\f32,\f52})}\Big).
 \end{split}
\end{equation*}
and
\begin{equation*}
 \begin{split}
\|\dl\mathcal{M}\|_{L_T^1(\tilde{B}^{-\f{1}{2},\f{1}{2}})}&\lesssim \|\dl\d\|_{L^\infty_T(\tilde{B}^{-\f12,\f12})}\|\nabla(\d_1-\hat{\d})\|^2_{L^2_T(\tilde{B}^{\f32,\f52})}\\
&\quad+\Big(1+\|\d_2-\hat{\d}\|_{L^\infty_T(B^{\f32})}\Big)\Big(\|\nabla(\d_2-\hat{\d})\|_{L_T^1(B^{\f32})}\\
&\qquad+\|\nabla(\d_1-\hat{\d})\|_{L_T^1(B^{\f32})}\Big)\|\nabla\dl\d\|_{L^\infty_T(\tilde{B}^{-\f12,\f12})}.
 \end{split}
\end{equation*}
Substituting those estimates back into \eqref{1531}, we eventually get
$$\|(\dl\r,\dl\u,\dl\d)\|_{\mathfrak{B}_T^{\f12}}\le Z(T)\|(\dl\r,\dl\u,\dl\d)\|_{\mathfrak{B}_T^{\f12}}$$
with
\begin{equation*}
 \begin{split}
Z(T)&=\exp\Big(C\|\u_2\|_{L^1_T(B^{\f52})}\Big)\Big[\|\r_1-1\|_{L^\infty_T(\tilde{B}^{\f12,\f32})}+\|\Dv\u_2\|_{L^1_T(B^{\f32})}+\|\u_2\|_{L^2_T(B^{\f32})}\\
&\quad+\|\nabla\u_1\|_{L^2_T(B^{\f12})}+\Big(1+\|\r_1-1\|_{L^\infty_T(B^{\f32})}+\|\r_2-1\|_{L^\infty_T(B^{\f32})}\Big)\|\nabla^2\u_2\|_{L^1_T(B^{\f12})}\\
&\quad+\Big(1+\|\r_1-1\|_{L^\infty_T(B^{\f32})}\Big)\|\nabla(\d_1-\hat{\d})\|^2_{L^2_T(\tilde{B}^{\f32,\f52})}\\
&\quad+\Big(1+\|\r_1-1\|_{L^\infty_T(B^{\f32})}+\|\r_2-1\|_{L^\infty_T(B^{\f32})}\Big)\||\nabla(\d_2-\hat{\d})|^2\|_{L^1_T(B^{\f32})}\\
&\quad+\Big(1+\|\r_1-1\|_{L^\infty_T(B^{\f32})}\Big)\Big(\|\d_1-\hat{\d}\|_{L^2_T(\tilde{B}^{\f32,\f52})}+\|\d_2-\hat{\d}\|_{L^2_T(\tilde{B}^{\f32,\f52})}\Big)\\
&\quad+\Big(1+\|\d_2-\hat{\d}\|_{L^\infty_T(B^{\f32})}\Big)\Big(\|\nabla(\d_2-\hat{\d})\|_{L_T^1(B^{\f32})}+\|\nabla(\d_1-\hat{\d})\|_{L_T^1(B^{\f32})}\Big)\Big].
 \end{split}
\end{equation*}
We notice that $\limsup_{T\rightarrow 0^+}Z(T)\le
C\|\r_1-1\|_{L^\infty(\tilde{B}^{\f12,\f32})}$. This is because all
other terms involve an integral in time in $L^1$ or $L^2$ sense so that as $T$ goes to zero, all those integrals will converge to
zero. Thus, if $\eta>0$ is sufficiently small, we get
$$\|(\dl\r,\dl\u,\dl\d)\|_{L^\infty_T(B^{\f12})}=0$$ for certain $T>0$
small enough.
Thus, we have shown the uniqueness on a small time interval $[0,T]$ such that $(\r_1,\u_1, \d_1)=(\r_2,\u_2, \d_2)$.

Then we can argue as in \cite{RD2} for the compressible Navier--Stokes equations.
Let $T_{max}<+\infty$ be the largest time such that
the two solutions coincide on $[0, T_{max}]$. Taking $T_{max}$ as the initial time, we denote
$$(\widehat{\r_i}(t), \widehat{\u_i}(t),\widehat{\d_i}(t))\overset{def}=(\r_i(t-T_{max}), \u_i(t-T_{max}),
\d_i(t-T_{max})).$$ Repeating the above arguments and using the fact
that $\|\widehat{\r_i}(0)-1\|_{L^\infty}\le\f{1}{4}$, we can prove that
$$(\widehat{\r_1}(t),\widehat{\u_1}(t),\widehat{\d_1}(t))=(\widehat{\r_2}(t), \widehat{\u_2}(t),\widehat{\d_2}(t))$$
on a sufficiently small interval $[0,\iota]$ with $\iota>0$. This
contradicts the assumption that $T_{max}$ is the largest time such
that the two solutions coincide. Thus, $T_{max}=+\infty$ which means that the
uniqueness result holds in $\R^+$.

\end{document}